\documentclass[hidelinks,onefignum,onetabnum]{siamart220329}

\usepackage{lipsum}
\usepackage{amsfonts}
\usepackage{graphicx}
\usepackage{epstopdf}
\usepackage{algorithmic}
\ifpdf
  \DeclareGraphicsExtensions{.eps,.pdf,.png,.jpg}
\else
  \DeclareGraphicsExtensions{.eps}
\fi

\newsiamremark{remark}{Remark}
\newsiamremark{assumption}{Assumption}
\crefname{assumption}{Assumption}{Assumptions}

\headers{SiMPL: Analysis}{B. Keith, D. Kim, B. S. Lazarov, and T. M. Surowiec}

\title{Analysis of the SiMPL method for density-based topology optimization\thanks{Current version as of \today.
\funding{
  This work was performed under the auspices of the U.S.\ Department of Energy by Lawrence Livermore National Laboratory under contract DE-AC52-07NA27344 and the LLNL-LDRD Program under Project tracking No.\ 22-ERD-009.
  Release number LLNL-JRNL-870167.
  BK, DK, and BL were partially supported by the LLNL-LDRD Program under Project Tracking No.\ 22-ERD-009.
  BK and DK were also supported in part by the U.S.\ Department of Energy Office of Science Early Career Research Program under Award Number DE-SC0024335.
    }
}}

\author{
  Brendan~Keith\footnotemark[2]
  \and
  Dohyun~Kim\thanks{Division of Applied Mathematics, Brown University, Providence, RI 02912, United States of America
  (\email{brendan\_keith@brown.edu}, \email{dohyun\_kim@brown.edu}).}
  \and Boyan~S.~Lazarov\thanks{Lawrence Livermore National Laboratory, Livermore, CA 94550, United States of America
  (\email{lazarov2@llnl.gov}).}
  \and Thomas~M.~Surowiec\thanks{Department of Numerical Analysis and Scientific Computing, Simula Research Laboratory, 0164 Oslo, Norway
  (\email{thomasms@simula.no}).}
}
\usepackage{amsopn}

\usepackage{amssymb}
\usepackage{mathtools}
\usepackage{tcolorbox}
\usepackage{multirow}
\usepackage{float}

\usepackage[titletoc]{appendix}

\usepackage{enumerate}
\usepackage{bm}
\newcommand{\argmin}{\mathop{\mathrm{argmin}}}

\newcommand{\sig}{\sigma}
\newcommand{\iprod}[2]{\langle #1, #2\rangle}
\DeclareMathOperator*{\esssup}{ess\,sup}
\DeclareMathOperator*{\essinf}{ess\,inf}
\DeclareMathOperator{\dd}{d\!}
\DeclareMathOperator{\C}{\mathsf{C}}

\usepackage{enumitem}

\usepackage[ruled,vlined,algo2e]{algorithm2e}

\crefname{equation}{}{}
\crefname{algocf}{algorithm}{algorithms}
\SetKwComment{Comment}{$\triangleright$\ }{}

\usepackage{accents}
\makeatletter
\newcommand{\sbullet}{%
  \hbox{\fontfamily{lmr}\fontsize{.6\dimexpr(\f@size pt)}{0}\selectfont\textbullet}}
\DeclareRobustCommand{\mathbullet}{\accentset{\sbullet}}
\makeatother

\usepackage{subfig}
\usepackage{tikzscale}
\usepackage{pgfplots}
\pgfplotsset{width=7cm,compat=1.8}
\usetikzlibrary{pgfplots.groupplots}{}

\definecolor{color0}{rgb}{0.7843, 0.7843, 0.7843}
\definecolor{color1}{rgb}{0, 0.4470, 0.7410}
\definecolor{color2}{rgb}{0.8500, 0.3250, 0.0980} \definecolor{color3}{rgb}{0.9290, 0.6940, 0.1250}
\definecolor{color4}{rgb}{0.7060, 0.3840, 0.7650}
\definecolor{color5}{rgb}{0.4660, 0.6740, 0.1880}
\definecolor{color6}{rgb}{0.3010, 0.7450, 0.9330}
\definecolor{color7}{rgb}{0.6350, 0.0780, 0.1840}
\definecolor{color8}{rgb}{0.0, 0.4078, 0.3412}

\ifpdf
\hypersetup{
  pdftitle={Analysis of the SiMPL method for density-based topology optimization},
  pdfauthor={B.\ Keith, D.\ Kim, B.\ S.\ Lazarov, and T.\ M.\ Surowiec}
}
\fi

\begin{document}

\setlength{\abovedisplayshortskip}{2.1pt}
\setlength{\belowdisplayshortskip}{2.1pt}
\setlength{\abovedisplayskip}{2.1pt}
\setlength{\belowdisplayskip}{2.1pt}
\captionsetup{belowskip=0pt}

\maketitle
\begin{abstract}
We present a rigorous convergence analysis of a new method for density-based topology optimization that provides point-wise bound preserving design updates and faster convergence than other popular first-order topology optimization methods.
	Due to its strong bound preservation, the method is exceptionally robust, as demonstrated in numerous examples here and in the companion article \cite{simplapp}.
	Furthermore, it is easy to implement with clear structure and analytical expressions for the updates.
	Our analysis covers two versions of the method, characterized by the employed line search strategies.
	We consider a modified Armijo backtracking line search and a Bregman backtracking line search.
	For both line search algorithms, our algorithm delivers a strict monotone decrease in the objective function and further intuitive convergence properties, e.g., strong and pointwise convergence of the density variables on the active sets, norm convergence to zero of the increments, convergence of the Lagrange multipliers, and more.
	In addition, the numerical experiments demonstrate apparent mesh-independent convergence of the algorithm. We refer to the new algorithm as the SiMPL method, pronounced like ``simple", which stands for \underline{Si}gmoidal \underline{M}irror descent with a \underline{P}rojected \underline{L}atent variable.
\end{abstract}

\begin{keywords}
	topology optimization, mirror descent, Bregman divergence, line search, nonconvex optimization.
\end{keywords}

\begin{MSCcodes}
	68Q25, 68R10, 68U05
\end{MSCcodes}

\section{Introduction}

Topology optimization is a powerful tool in engineering design, enabling the discovery of optimized, efficient, and innovative structures and components. The technique finds an optimized material layout within a given design space for a prescribed set of loads, boundary conditions, and constraints. The overall goal is to maximize the performance of a structure, component, or the entire physical system. There is a wealth of literature on the subject, and we point the interested reader to the well-known monographs \cite{Bendsoe1995,Bendsoe2004} as a starting point. Further key references are provided throughout the text below.

We note here that, while second-order optimization methods have been developed for topology optimization \cite{Rojas-Labanda2016, Adam2018, Baraldi2023}, first-order methods remain of significant interest and are the most widely used in academia and industry \cite{Svanberg1987, Fleury1989, Zhou1991, Svanberg2002}. Second-order methods converge in fewer (outer) iterations. However, the computational cost per iteration is significantly higher than first-order methods, which, in contrast, offer simple algorithmic implementations, thereby further strengthening their popularity.

The main aim of this paper is to analyze a simple and efficient new first-order method for density-based topology optimization: \underline{Si}gmoidal \underline{M}irror descent with a \underline{P}rojected \underline{L}atent variable or, simply, the SiMPL method.
The overall goal is to provide the SiMPL method with a rigorous mathematical foundation.
The companion article~\cite{simplapp} focuses on applications of the SiMPL method and its computer implementation. Nevertheless, we do include several illustrative numerical examples at the end of this paper. The two numerical examples illustrate the advantages of a first-optimize-then-discretize approach, not only in terms of the formulation of the algorithm and corresponding convergence analysis, which extends to any conforming discretization, but also by demonstrating mesh and order independence.

\subsection{Background}
\label{sub:Background}

In an ideal setting, the optimal material layout would be represented by a binary-valued mapping over the design domain $\Omega \subset \mathbb{R}^d$, where $0$ represents the void (no material) and $1$ indicates the placement of material.
However, this would result in an essentially intractable infinite-dimensional mixed integer program and lead to a so-called ``checkerboard'' design when implemented.
A popular solution to this issue is to relax the binary requirement and search for  a material distribution $\rho $ with $\rho(x) \in [0,1]$ at each point $x \in \Omega$.
This is called density-based topology optimization.
Somewhat surprisingly, this relaxation is not enough to guarantee the existence of a solution on a continuous level.
One possibility to guarantee the existence of a solution is to convolve the design density $\rho$ with a filter function \cite{Bourdin2001}.
Filtering the density $\rho$ introduces enough compactness into the continuous problem to ensure convergence of the filtered density $\tilde{\rho}= \tilde\rho(\rho)$ in a sufficiently regular function space.
It also has the practical side effect of removing the non-physical checkerboard patterns \cite{lazarov2011-filter} seen in the early topology optimization literature. This approach is referred to as the two-field density representation.

The mathematical model governing the mechanical behavior of the system is often a linear elliptic partial differential equation (PDE) defined over the entire design domain $\Omega$. The decision variable $\rho$ then enters the consitutive law of this PDE through an interpolation of the filtered density $\tilde\rho$ between material and void states. Unfortunately, interpolating to $0$ (complete void) is not appropriate since having subsets with no material can affect the well-posedness of the governing PDE. The usual remedy is setting a minimal, artificial material density $0 < \rho_0 \ll 1$ throughout the entire design domain. Furthermore, experience shows that linearly interpolating between $\rho_0$ and $1$ will usually not promote a binary design. One solution is to transform the filtered density $\tilde\rho$ by a nonlinear material interpolation law $r=r\left(\tilde\rho\right)$, which penalizes intermediate values, thus promoting a zero/one design \cite{Bendsoe1999}. The so-called Solid Isotropic Material with Penalization (SIMP) \cite{Bendsoe1989}, utilized here, is the most popular approach to steer the topology-optimized solution to a such a design state and has widely replaced the original homogenization formulation \cite{Bendsoe1988}. Given an appropriately chosen objective function, we can now formulate the task at hand as a meaningful PDE-constrained optimization problem.

The most popular numerical solution algorithms for topology optimization have been developed on the discrete level.
In other words, the PDE-constrained optimization problem is first discretized via, e.g., the finite element method, and then solved with various optimization algorithms, often used as black box solvers.
The optimality criteria method (OC) \cite{Bendsoe1995, Yin2001} and the Method of Moving Asymptotes (MMA) \cite{Svanberg1987} are the most widely used.
Despite being based on heuristic arguments and limited mathematical foundations \cite{Ananiev2005}, OC gained popularity due to its ease of implementation \cite{Sigmund2001}.
On the other hand, MMA became popular due to its flexibility, robust behavior, and openly available serial and parallel implementations; see, e.g., the popular open source Matlab \cite{Svanberg2014MMAAG} (serial) and C++ PETSc-based \cite{Aage2014} (parallel) implementations.

Our goal here is to approach the problem from a different angle.
Rather than first discretize and then optimize, we present an algorithm based on the latent variable proximal point (LVPP) framework for solving variational problems with pointwise constraints in continuous settings, which was proposed in \cite{keith2023proximal}.
LVPP is based on the classical proximal point method \cite{Moreau1965, rockafellar1976monotone, Parikh2014, martinet1970, Teboulle2018} with the pivotal innovation being the introduction of a so-called \emph{latent variable} to obtain subproblems that are convenient for discretization.
The latent variable generates a pointwise feasible discrete solution via an alternative solution representation described in \Cref{sub:LVMD} below.

Keith and Surowiec \cite{keith2023proximal} argued formally that the LVPP framework can be applied to topology optimization via mirror descent \cite{nemirovskij1983problem}, which is a popular first-order variant of the proximal point method \cite{beck2003mirror, Teboulle2018}.
The resulting algorithm displayed promising features such as a bound-preserving discrete solution, mesh-independence, and an easy computer implementation.
However, it utilizes an ad-hoc step size rule, $\alpha_k = ck$, which caused the convergence rate and the final design to depend significantly on the tunable parameter $c >0$; cf.~\cite[Section~6.4]{keith2023proximal}.

\subsection{Contributions}
In this work, we rigorously analyze the topology optimization algorithm referred to as \underline{Si}gmoidal \underline{M}irror descent with a \underline{P}rojected \underline{L}atent variable, a.k.a.\ the SiMPL method, given in \Cref{alg:SiMPL} below.
The standard theory of mirror descent \cite{nemirovskij1983problem} and Bregman proximal methods more generally, e.g., \cite{bauschke2001essential}, is carried out in \emph{reflexive} Banach spaces.
In general, the analysis of the SiMPL method requires specialized techniques beyond this setting.
For instance, the Bregman divergence used to derive the SiMPL method is only strongly convex in the \emph{non-reflexive} space $L^1(\Omega)$ \cite{Borwein1991}. 
Thus, beginning with the existence and uniqueness of the iterates established in \Cref{thm:entropy-projection}, we must often work with \emph{non-reflexive} norm topologies, weak topologies, and function spaces of extended real-valued functions on $\Omega$.
This leads to several technical contributions within the presented theory that reveal the intricate aspects of optimization in infinite-dimensional spaces.

\vspace*{-0.5em}
\begin{algorithm2e}[ht]
	\DontPrintSemicolon
	\caption{\label{alg:SiMPL}
		The SiMPL method for topology optimization.
	}

	\SetKwInOut{Function}{internal function\ }
	\SetKwInOut{Input}{input}
	\SetKwInOut{Output}{output}
	\BlankLine
	\Function{sigmoid function $\sig(\psi) \coloneqq \dfrac{\exp \psi}{\exp \psi + 1}$.}
	\BlankLine

\Input{initial latent variable $\psi \in L^\infty(\Omega)$.}
	\BlankLine
	\Output{optimized material density $\rho = \sig(\psi)$.}
	\BlankLine
	\BlankLine

\While{true}
	{
		Solve for $\mu \in \mathbb R$ such that $\int_\Omega \sig(\psi - \mu) \dd x = \theta |\Omega|$.
		\Comment*[r]{See \Cref{rem:mu-solve}}
		Correct the latent variable $\psi \leftarrow \psi - \mu$.\;
		Update the design density $\rho \leftarrow \sig(\psi)$.\;
		\If{a convergence test is satisfied}
		{
			\Return $\rho$.\;
		}
		Compute the gradient of the objective function $F$ with respect to $\rho$.\;
		Select the step size $\alpha > 0$.
		\Comment*[r]{ See \Cref{sec:adaptive}}
		Increment the latent variable $\psi \leftarrow  \psi - \alpha \nabla F(\rho)$.\;
	}

\end{algorithm2e}
\vspace*{-0.5em}

\subsection{Notation}
Throughout this work, $\Omega$ denotes an open bounded domain in $\mathbb{R}^d$, where $d = 2$ or $3$, with a Lipschitz boundary $\partial \Omega$.
Likewise, $L^p(\Omega)$ denotes the usual Lebesgue space of $p$-integrable functions when $1\leq p<\infty$, and essentially bounded functions when $p=\infty$, respectively.
The Sobolev space of $L^p$ functions with $L^p$-integrable weak derivatives is denoted by $W^{1,p}(\Omega)$.
We denote the $L^2(\Omega)$-inner product of two functions $u, v$ as $(u,v) = \int_\Omega uv \dd x$.
When $p=2$, we denote $H^1(\Omega)=W^{1,2}(\Omega)$.

The norm on a Banach space $X$ is denoted by $\|\cdot\|_X$ and the canonical pairing with its topological dual $X'$ is denoted by $\iprod{\cdot}{\cdot}_{X',X}$.
Whenever it is clear from context, we omit the subscripts.
For sequences of elements $x_k \in X$, indexed in by $k\in\mathbb{N}$, we simply write $\{x_k\}$.
The Lebesgue measure of a measurable subset $\mathcal{B} \subset \mathbb R^d$ is simply denoted by $|\mathcal{B}|$.
Let $\varphi\colon X \to \mathbb{R}\cup\{+\infty\}$ be a proper function.
We use $\operatorname{dom} \varphi \coloneqq \{ x \in X : \varphi(x) \in \mathbb{R}\}$ to denote the effective domain of $\varphi$.
Moreover, if $\varphi$ is smooth on the interior of its effective domain, then we use $\varphi^\prime \colon \operatorname{int}\operatorname{dom} \varphi \to X^\prime$ to denote the corresponding Fr\'echet derivative.

Throughout the text, a central role is played by the constraints on the density $\rho$. Depending on the setting, we need to view these constraints through different function spaces. Accordingly, we
define the following sets:
\begin{align*}
	L^p_\theta(\Omega)  & \coloneqq
	\Big\{\rho\in L^p(\Omega) : \int_\Omega\rho \dd x =  \theta|\Omega|\Big\}\,,
	\\
	L^p_{[0,1]}(\Omega) & \coloneqq \{\rho\in L^p(\Omega) : 0\leq\rho\leq1\text{ a.e.\ in }\Omega\}\,,
\end{align*}
where $1 \leq p \leq \infty$ distinguishes the norm topology and $0\leq\theta\leq1$ is a fixed parameter denoting the volume fraction of the desired design.
We elect to consider the volume constraint as an equality as we require the use of the Implicit Function Theorem later in the text. Much of the results on the update formula, etc.\ remain intact if we choose to use an inequality constraint.
We emphasize that $L^p_{[0,1]}(\Omega)$ does not change with $p$, but that it will have different norm topologies.
An important consequence is that the interior of the latter set vanishes, i.e., $\operatorname{int}L^p_{[0,1]}(\Omega) = \emptyset$, when $1\leq p <\infty$.
However,
\[
	\operatorname{int}L^\infty_{[0,1]}(\Omega)=\{\rho\in L^\infty(\Omega) : \essinf \rho > 0 ~\text{ and } \esssup \rho < 1\}
	\,,
\]
as can be readily verified.
Next, we define the set of admissible densities,
\begin{equation}
	\label{eq:AdmissibleSet}
	\mathcal{A}
	\coloneqq
	L^p_\theta(\Omega) \cap L^p_{[0,1]}(\Omega)
	\,,
\end{equation}
where we have intentionally not included a superscript-$p$ so as not to distinguish among its different possible topologies.
Finally, for later convenience we denote the set of regularized densities by
\begin{equation}
	\label{eq:RegularizedSet}
	\mathbullet{\mathcal{A}}
	\coloneqq
	L^\infty_\theta(\Omega) \cap \operatorname{int} L^\infty_{[0,1]}(\Omega)
	\,,
\end{equation}
within which we will show that every design density iterate of the SiMPL method, \Cref{alg:SiMPL}, evolves (cf.\ \Cref{thm:update-rule}).

\subsection{Outline}
The rest of the text is organized as follows. After briefly introducing further notation and some preliminary results from analysis, we turn to a general optimization framework in \Cref{sec:conti}.
Here, we investigate the effects of the PDE-filter and the regularity of the gradient. In \Cref{sec:mirror-descent}, we derive the SiMPL method. This involves a quick review of mirror descent
and is accompanied by a result about analytical properties of the Fermi--Dirac entropy on $L^p$-spaces, which are essential for the subsequent analysis. With these
tools at hand, we rigorously derive a primal-dual update rule and finally the SiMPL method as a latent variable update rule. Due to the structure of the feasible set and
choice of entropy, we can immediately prove a number of intuitive properties of the primal and latent iterates and connections to a classical KKT-system. This is done in \Cref{sec:prelim-conv}.
Afterwards, in \Cref{sec:adaptive}, we propose to augment SiMPL with
two well-known line search strategies. We motivate the choice of the initial step size for the line searches using the concept of relative smoothness. We close the
theoretical portion of the text with the analysis of the convergence behavior of these two globalized versions of SiMPL. In particular, we prove that SiMPL will generate
sequences of iterates that strictly and monotonically decrease the objective function values as well as prove norm convergence of the increments. We end the technical part of the paper with \Cref{sec:num} by putting the
theory to test in several numerical experiments, where we observe mesh-independent behavior.

Our findings are positive overall and summarized in \Cref{sec:Conclusion}.

\section[A general framework for topology optimization]{A general framework for topology optimization\nopunct}\label{sec:conti}comprises several key components.
The following were described in \Cref{sub:Background} of the introduction:

\begin{enumerate}[leftmargin=2em,itemsep=3pt,topsep=3pt,label=(\roman*)]
\item We solve an optimization problem that provides an optimal density $\rho \in \mathcal{A}$.
\item Practical/mathematical reasons necessitate the use of a filtered density, $\tilde{\rho} = \tilde{\rho}(\rho)$.
\item For well-posedness, we interpolate $\tilde{\rho}$ between a void density $\rho_0>0$ and $1$.
\item The interpolated filtered density controls the solution of a PDE.
\end{enumerate}

Following \cite{lazarov2011-filter}, we consider a filter defined by a screened Poisson equation:
\begin{align}\label{eq:filt-eq}
\begin{split}
-\epsilon^2 \Delta \tilde{\rho} + \tilde{\rho} & =\rho~ \text{ in }\Omega,         \\
\nabla \tilde{\rho}\cdot{n}     & =0~ \text{ on }\partial\Omega.
\end{split}
\end{align}
The constant $\epsilon=r_{\min}/2\sqrt{3}>0$ enforces a length scale to $\tilde{\rho}$ via the parameter $r_{\min}>0$. It follows from standard results on linear elliptic boundary value problems that $\tilde{\rho}$ can be viewed as an bounded linear operator of $\rho$ (e.g., an injection from $L^2(\Omega)$ into $H^1(\Omega)$). We gather additional properties in a lemma below. This operator viewpoint means we do not need to include \eqref{eq:filt-eq} as a constraint in the optimization problem.

We choose the Helmholtz-type filter as it aligns naturally with our functional framework in the optimization process, particularly since we operate in a Sobolev space and aim for high-order discretization.
However, the SiMPL method is not restricted to this choice and can also accommodate other filtering techniques, such as finite-support convolution-based filters, providing sufficient compactness and regularity properties (cf.\ \Cref{lem:filter-bdd} and \Cref{as:complete_continuity}).

We assume throughout that the appropriate physics for the final filtered design $\tilde{\rho}$ follow the well-known equations of linear elasticity.
That is, given a symmetric positive-definite elasticity tensor $\C \colon \mathbb{R}^{d\times d}_{\mathrm{sym}} \to \mathbb{R}^{d\times d}_{\mathrm{sym}}$, an internal body force density
$f \colon \Omega \to \mathbb{R}^d$, and a filtered, interpolated density $r(\tilde{\rho})$, 
we model the displacement $u: \Omega \to \mathbb R^d$ of the final design $\tilde{\rho}$ as the solution of 
\begin{align}\label{eq:lin-elas}
\begin{split}
-\operatorname{div}(r(\tilde{\rho}) \C{\varepsilon}({u})) &= {f}~ \text{ in }\Omega,\\
u &= 0~ \text{ on } \Gamma_0,\\
 (\C{\varepsilon}({u}))\cdot n &= 0~ \text{ on } \partial \Omega \setminus \Gamma_0.
\end{split}
\end{align}
Here, ${\varepsilon}({u}) \coloneqq (\nabla u + (\nabla u)^\top)/2$ is the symmetric gradient,  $\Gamma_0 \subset \partial \Omega$ denotes a subset of the domain boundary $\partial \Omega$ with positive Hausdorff measure, $n$ denotes the unit outward normal vector on $\partial \Omega$, and $r(\cdot)$ is the popular SIMP interpolation model, 
\[
r(\tilde{\rho})=\rho_0 + (1-\rho_0)\tilde{\rho}^s,
\] 
with $0<\rho_0\ll 1$ denoting a minimum void density and $s>1$ denoting the penalization exponent.
Typical parameter values are $\rho_0=10^{-6}$ and $s=3$; cf.~\cite{Bendsoe1995}.

We will consider a reduced space problem formulation, as in, e.g., \cite{Manzoni2021}, by viewing the displacement $u$ as a function of $\rho$. This allows us to also avoid posing the system \eqref{eq:lin-elas} as explicit constraints in the optimization problem. 
In particular, given a sufficiently smooth objective function $\widehat{F}$ that acts on $\rho$, $\tilde{\rho}$, and $u$, we can use the assumptions and observations made above to consider a so-called reduced objective function, written solely as a function of the density $\rho$: namely, $F(\rho) := \widehat{F}(\rho,\tilde{\rho}(\rho),u(\tilde{\rho}(\rho)))$. This allows us to formulate the topology optimization problem in a compact form:
\begin{equation}\label{eq:red-top-opt}
\min_{\rho \in \mathcal{A}} F(\rho)
.
\end{equation}
The following two subsections gather useful results on the screened Poisson filter and the structure of the gradient of the reduced objective $F(\rho)$.

\subsection{Filtering preserves the admissible set}
We find by \cite[Proposition 9.30]{Brezis2011} (in the case of a smooth boundary) or \cite[Theorem 2.6]{Kim2020} (for a general Lipschitz domain), that the filter equation \eqref{eq:filt-eq} preserves the bounds and volume of $\rho$.
\begin{lemma}\label{lem:filter-bdd}
	Let $\Omega\subset\mathbb{R}^d$ be a Lipschitz domain and $g\in L^q(\Omega)$ with $q>d/2$.
	Let $w\in H^1(\Omega)$ be the weak solution to the following equation:
	\begin{align*}
		-\epsilon^2 \Delta w + w =g ~~~ \text{in }\Omega,
		\qquad
		\nabla w\cdot{n} =0 ~~~ \text{on }\partial\Omega.
	\end{align*}
	Then there exists a constant $c>0$, depending only on $\epsilon>0$ and the domain $\Omega$, such that
	\begin{align*}
		\|w\|_{L^\infty(\Omega)} \leq c\|g\|_{L^q(\Omega)}.
	\end{align*}
	Moreover, $\essinf g\leq w \leq \esssup g\text{ a.e.\ in }\Omega$ and $\int_\Omega w\dd x = \int_\Omega g\dd x$.
\end{lemma}

\Cref{lem:filter-bdd} shows that $\rho \in \mathcal{A}$ implies $\tilde\rho \in \mathcal{A}$.
In particular, we are guaranteed that $\essinf \tilde\rho \geq 0$, thus $r(\tilde{\rho}) \geq \rho_0 > 0$ a.e.\ in $\Omega$.
This implies that the linear elasticity equation~\cref{eq:lin-elas} is strongly elliptic for the given boundary conditions and admits a unique weak solution ${u}\in V\coloneqq\{{v}\in [H^1(\Omega)]^d: {v}|_{\Gamma_0}=0\}$.
At the discrete level, the maximum principle may not hold, particularly when high-order approximations are used; cf.\ \Cref{rem:filter-dmp}.
\subsection{The gradient is essentially bounded}

We require the gradient $\nabla F$ to be essentially bounded to guarantee that the iterates of~\Cref{alg:SiMPL} belong to $\mathbullet{\mathcal{A}}$; cf.~\Cref{thm:update-rule}.
This can be achieved under mild assumptions on the domain $\Omega$ and data $f$.

The following lemma is immediate from the chain rule; see \cite[Appendix A]{simplapp}.
\begin{subequations}
\begin{lemma}\label{lem:gradF}
	For a given design density $\rho\in \mathcal{A}$, let $\tilde{\rho}\in H^1(\Omega)$ and $u \in V$ be the corresponding solutions to \cref{eq:filt-eq,eq:lin-elas}, resectively.
  If we assume that $\widehat{F}(\tilde{\rho},u)$ is Fr\'{e}chet differentiable with respect to $\tilde{\rho} \in H^1(\Omega)$ and $u\in [H^1(\Omega)]^d$,
	then the reduced objective function $F$ is Fr\'{e}chet differentiable in $L^\infty(\Omega)$ and its derivative can be characterized by the variational equation
	\begin{equation}
	\label{eq:frechet_derivative}
		\iprod{F'(\rho)}{q} =
		\int_\Omega \nabla F(\rho) q \dd x ~\text{ for all } q\in L^\infty(\Omega)
		\,,
	\end{equation}
	where $\nabla F(\rho) := \tilde{\lambda}\in H^1(\Omega)$ solves the adjoint filter equation
	\begin{multline}\label{eq:adj_filter}
		\int_\Omega \big(\epsilon^2\nabla\tilde{\lambda}\cdot\nabla\tilde{q} + \tilde{\lambda}\tilde{q}\big)\dd x
		=
		\Big\langle \frac{\partial}{\partial {\tilde{\rho}}}\widehat{F}(\tilde{\rho},u), \tilde{q} \Big\rangle
		-
		\int_\Omega r'(\tilde{\rho}){\varepsilon}({\lambda}):(\C{\varepsilon}({u}))\,\tilde{q}\dd x
		\\~\text{ for all } \tilde{q}\in H^1(\Omega)
		\,,
	\end{multline}
	and ${\lambda}\in V$ solves the adjoint state equation
	\begin{equation}\label{eq:adj}
		\int_\Omega r(\tilde{\rho})(\C{\varepsilon}({v})):{\varepsilon}({\lambda})\dd x
		=
		\Big\langle \frac{\partial}{\partial {u}}\widehat{F}(\tilde{\rho},u), v \Big\rangle
		~\text{ for all } {v}\in V
		\,.
	\end{equation}

\end{lemma}
\end{subequations}

\Cref{lem:gradF} shows that the Fr\'echet derivative $F'(\rho)$ admits a primal representative $\nabla F(\rho) \in H^1(\Omega)$ at each point $\rho \in \mathcal{A}$.
	Under sufficient regularity assumptions, we can further conclude that
	\begin{equation}
		\nabla F(\rho) \in L^\infty(\Omega)
		\,.
	\end{equation}
	Indeed, following \cite[Theorem 2.1]{Bensoussan2002}, if $\Omega$ satisfies the exterior sphere condition \cite{Bensoussan2002} and $\partial_{{u}}\widehat{F}, f\in L^2(\Omega)$, then there exists $\delta > 0$ that depends only on $\Omega$ and $f$, such that
  \begin{align*}
    {u},{\lambda}\in [W^{1,2 + \delta}(\Omega)]^d\,.
  \end{align*}
	Furthermore, provided  $\delta>d-2$, \cite[Theorem 2.6]{Bensoussan2002} implies that $\tilde{\lambda}\in L^\infty(\Omega)$ and, in particular, satisfies the following estimate:
	\begin{align*}
    \|\tilde{\lambda}\|_{L^\infty(\Omega)}\leq c\|{\varepsilon}({u})\|_{L^{1+\delta/2}(\Omega)}\|{\varepsilon}({\lambda})\|_{L^{1+\delta/2}(\Omega)}\,.
	\end{align*}
	When $d=2$, $\delta>0$ is enough to ensure that $\tilde{\lambda}\in L^\infty(\Omega)$.
  When $d=3$, we need $\delta > 1$ and hence require stronger assumptions on the domain and data, e.g., convexity of $\Omega$ and $f,\partial_u\widehat{F}\in L^2(\Omega)$.

\section{Deriving SiMPL}\label{sec:mirror-descent}
In this section, we rigorously derive the SiMPL method for topology optimization.

At its core, SiMPL is a two-step algorithm composed of an unconstrained gradient update followed by a volume correction; cf.~\Cref{alg:SiMPL}.
In the gradient update step, one finds a bound-preserving intermediate density in $\operatorname{int} L^\infty_{[0,1]}(\Omega)$.
This intermediate density is then projected to the admissible set, resulting in a volume-corrected density in the $L^\infty(\Omega)$-relative interior of $\mathcal{A}$; i.e., $\mathbullet{\mathcal{A}}$.

The outline of this section is as follows. We briefly review the classical mirror descent method, from which we ultimately derive SiMPL. 
This is followed by a rigorous analysis of the Fermi--Dirac entropy on various $L^p$-spaces. We then derive a primal-dual update rule for the density $\rho$ 
and Lagrange multiplier $\mu$ associated with the equality (volume) constraint and finally, after introducing the latent variable, we arrive at SiMPL. The benefits of 
using a latent variable from the perspective of using higher order discretizations are discussed at the end of the section.

\subsection{Bregman divergences and mirror descent}
The mirror descent method of Nemirovksij and Yudin \cite{nemirovskij1983problem} seeks to minimize a convex Lipschitz function over a reflexive Banach space by first constructing an update rule viewed as a mirror trajectory on the dual space.
This mirror trajectory is then mapped back to the primal space, leading to a nonlinear update rule for the decision variable. 
Our setting forces us to work in non-reflexive spaces and allows us construct a mirror trajectory on a latent function space.
However, the philosophy of SiMPL is indebted to the original mirror descent viewpoint.

The modern interpretation, e.g., \cite{Chen1993}, is to view mirror descent as a proximal gradient algorithm in which the proximal term is given by a Bregman divergence that aptly captures the geometry of the feasible set and/or domain of the objective functional.
For this reason, we recall some basic properties of Bregman divergences here.
The connection back to the original idea of Nemirovskij and Yudin is through the distance generating functional (entropy) used to define
the Bregman divergence.

Let $X$ be a Banach space.
The Bregman divergence \cite{Bregman1967TheRM} corresponding to an essentially smooth, essentially strictly convex function $\varphi \colon X \to \mathbb{R}\cup\{+\infty\}$ is defined
\begin{equation}
	D_\varphi(\rho,q) \coloneqq \varphi(\rho) - \varphi(q) - \iprod{\varphi'(\rho)}{\rho - q},
\end{equation}
for all $\rho \in \operatorname{dom}\varphi$ and $q \in \operatorname{int}\operatorname{dom}\varphi$.

Given that it captures the error in the first-order Taylor series expansion of $\varphi$, the Bregman divergence may be viewed as a generalization of a squared distance.

\begin{proposition}[Properties of Bregman divergences \cite{Chen1993}]
	\label{prop:Bregman}
	Let $\varphi \colon X \to \mathbb{R}\cup\{+\infty\}$ be essentially smooth and essentially strictly convex.

	The Bregman divergence induced by $\varphi$ satisfies the following properties:
	\begin{itemize}[leftmargin=1.5em]
		\item (Positivity) $D_\varphi(\rho,q)\geq 0$ for all $(\rho,q)\in \operatorname{dom}\varphi\times \operatorname{int} \operatorname{dom}\varphi$, where the equality holds iff $\rho=q$.
		\item (Convexity) $D_\varphi(\rho,q)$ is essentially strictly convex in its first argument.
		      Moreover, if $\varphi$ is strongly convex, then so is the map $\rho \mapsto D_\varphi(\rho,q)$.
	\end{itemize}

\end{proposition}

\subsection[The Fermi--Dirac entropy on Lp-spaces]{The Fermi--Dirac entropy on $L^p$-spaces}
In this work, we are interested in a specific function $\varphi \colon L^p(\Omega)\rightarrow\mathbb{R}\cup\{+\infty\}$ known as the (negative) Fermi--Dirac entropy,
\begin{equation}
	\label{eq:Fermi-Dirac}
	\varphi(\rho)=\begin{cases}
		\displaystyle
		\int_\Omega \rho\ln \rho + (1-\rho)\ln(1-\rho) \dd x & \text{ when }\rho\in L^p_{[0,1]}(\Omega), \\
		+\infty                                              & \text{ otherwise}.
	\end{cases}
\end{equation}
Here and throughout, we use the convention $0\ln0=0$.
Our interest in the Fermi--Dirac entropy lies in the fact that the principal bound constraint of density-based topology optimization, i.e., $0 \leq \rho \leq 1$ a.e.\ in $\Omega$, is encoded in its effective domain.
The following lemma lists several properties of $\varphi$.

\begin{lemma}[Properties of the Fermi--Dirac entropy]\label{lem:fermi-dirac}

	The Fermi--Dirac entropy, $\varphi \colon L^p(\Omega)\to \mathbb{R}\cup \{+\infty\}$, 
	is:
	\begin{enumerate}[leftmargin=2em]
		\item lower semicontinuous and strictly convex on $L^p_{[0,1]}(\Omega)$ for all $p\in [1,\infty]$;
		\item strongly convex on $L^1_{[0,1]}(\Omega)$;
		\item continuous on $L^p_{[0,1]}(\Omega)$ for all $p\in (1,\infty]$;
		\item continuously Fr\'{e}chet differentiable on $\operatorname{int} L^\infty_{[0,1]}(\Omega)$.
		      Moreover,
		      \begin{align*}
			      \langle \varphi^\prime(\rho), v \rangle = \int_\Omega \sig^{-1}(\rho) v \dd x ~\text{ for all } \rho \in \operatorname{int} L^\infty_{[0,1]}(\Omega),~ v\in L^1(\Omega)
			      \,,
		      \end{align*}
		      where $\sig:\mathbb{R}\rightarrow(0,1)$ is the (logistic) sigmoid function $\sig(\psi)=1/(1+\exp(-\psi))$ with the inverse $\sig^{-1}(\rho) = \ln(\rho/(1-\rho))$.

\end{enumerate}
\end{lemma}

\begin{proof}
	The lemma is immediate from \cite[Theorem 4.1]{keith2023proximal}, \cite[Theorem 2.7]{Borwein1991} and $\varphi(\rho) = \hat{\varphi}(\rho) + \hat{\varphi}(1-\rho) + 1$, where $\hat{\varphi}$ is the negative differential Shannon entropy $\hat{\varphi}(\rho) \coloneqq \int_\Omega \rho \ln \rho - \rho \dd x$.
\end{proof}

Among other things, this lemma makes it straightforward to show that the Bregman divergence associated to the Fermi--Dirac entropy is
\begin{align*}
	D_\varphi(\rho,q) = \int_\Omega \rho\ln\frac{\rho}{q} + (1-\rho)\ln\frac{1-\rho}{1-q} \dd x
	\,,
\end{align*}
for all $\rho \in L^p_{[0,1]}(\Omega)$ and $q \in \operatorname{int} L^\infty_{[0,1]}(\Omega)$.
We conclude this subsection with the following abstract result employing~\Cref{lem:fermi-dirac}.
This result provides the theoretical justification for an essential building block of SiMPL.

\begin{lemma}
	\label{thm:entropy-projection}
	Let ${A} \colon L^1(\Omega) \to \mathbb{R}$ be a bounded linear functional, $\alpha > 0$ a constant, $\rho \in \operatorname{int} L^\infty_{[0,1]}(\Omega)$ fixed, and $\mathcal{C}$ any non-empty, closed, convex subset of $L^1_{[0,1]}(\Omega)$.

	Then the following problem has a unique solution:
	\begin{equation}
		\label{eq:shift-entropy-projection}
		\min_{q \in \mathcal{C}}
		\left\{
		J(q) \coloneqq
		\langle A, q \rangle
		+ \alpha^{-1} D_{\varphi}(q,\rho)
		\right\}.
	\end{equation}
\end{lemma}

\begin{proof}
	By definition, we have
	\begin{align*}
		J(q) = \langle A, q \rangle + \alpha^{-1}\int_\Omega q\ln q-q\ln \rho+(1-q)\ln(1-q)-(1-q)\ln(1-\rho)\dd x .
	\end{align*}
	Since $\ln \rho$, $\ln(1-\rho) \in L^{\infty}(\Omega)$ we can assume $\alpha = 1$, shift the terms $-\int_\Omega q\ln \rho\dd x$ and $\int_\Omega q\ln(1-\rho)\dd x$ into $\langle A, q \rangle$, and ignore the constant $-\int_\Omega \ln(1-\rho)\dd x$.
	Thus, without loss of generality, we choose to rewrite
	\(
	J(q) = \langle A, q \rangle + \varphi(q).
	\)

	We can readily show that the set of functions $L^1_{[0,1]}(\Omega)$ is uniformly integrable. Then by the Dunford--Pettis theorem \cite[Theorem~2.4.5]{attouch2014variational}, $L^1_{[0,1]}(\Omega)$ is weakly sequentially compact in $L^1(\Omega)$.
	The same conclusion can be made for $\mathcal{C}$ in the hypothesis.

Let $\{q_k\} \subset \mathcal{C}$ be a minimizing sequence of $J$.
	Weak sequential compactness implies that the sequence $\{q_k\}$ admits a subsequence that is $L^1(\Omega)$-weakly convergent to some $\overline{q} \in \mathcal{C}$.
	Moreover, by \Cref{lem:fermi-dirac}, $J$ is weakly lower semicontinuous and $L^1(\Omega)$-strongly convex.
	Thus, we can use \cite[Theorem~3.2.1]{attouch2014variational} and strong convexity to deduce that $\overline{q}$ is the unique optimal solution.
\end{proof}

\subsection{A primal-dual update rule}
Let us recall the update rule for the projected gradient descent method with step size $\alpha > 0$.
Beginning from an initial guess $\rho_0 \in L^2_{[0,1]}(\Omega)$, this method constructs a minimizing sequence of design densities by optimizing a penalized first-order approximation of $F$ at each previous iteration:
\begin{align*}
	\rho_{k+1} \coloneqq \argmin_{\rho\in \mathcal{A}}\bigg\{\int_\Omega \nabla F(\rho_k)\rho \dd x + \frac{1}{2\alpha}\|\rho-\rho_k\|_{L^2(\Omega)}^2\bigg\}
	\,,
	\quad
	k = 0,1,2,\ldots
\end{align*}
Solving these subproblems is equivalent to applying an update rule
\begin{equation}\label{eq:pgd-org}
	\rho_{k+1} = \mathcal{P}_{L^2}(\rho_k - \alpha\nabla F(\rho_k)),
\end{equation}
from which the term projected gradient descent derives.
More specifically, $\mathcal{P}_{L^2}:L^2(\Omega)\rightarrow \mathcal{A}$ is a canonical $L^2(\Omega)$-projection operator such that for all $v\in L^2(\Omega)$
\begin{subequations}\label{eq:L2proj}
	\begin{align}
		\mathcal{P}_{L^2}(v)=\max(0, \min(1, v - \mu))\text{ a.e. in }\Omega,
		\intertext{ where $\mu\in\mathbb{R}$ solves the volume correction equation }
		\int_\Omega \max(0, \min(1, v - \mu))\dd x=\theta|\Omega|.
	\end{align}
\end{subequations}
Denoting the argument of the projection operator in~\eqref{eq:pgd-org} by $\rho_{k+1/2}$, this update rule can be decomposed into an unconstrained gradient descent step followed by projection onto the admissible set:
\begin{subequations}\label{eq:pgd}
	\begin{align}
		\rho_{k+1/2} & = \rho_k - \alpha \nabla F(\rho_k), \\
		\rho_{k+1}   & = \mathcal{P}_{L^2}(\rho_{k+1/2}).
	\end{align}
\end{subequations}

By replacing the squared $L^2$ distance $\frac{1}{2}\|\cdot\|_{L^2(\Omega)}^2$ with a Bregman divergence, we obtain a more general update rule:
\begin{equation}
	\label{eq:mirror-descent-minimizer}
	\rho_{k+1}=\argmin_{\rho\in \mathcal{A}}\Big\{J_k(\rho) \coloneqq \int_\Omega \nabla F(\rho_k)\rho \dd x + \frac{1}{\alpha}D_\varphi(\rho, \rho_k)\Big\}.
\end{equation}
In our setting, $\varphi$ is the Fermi--Dirac entropy introduced above. 
The following theorem establishes the existence of each $\rho_{k+1}$ along with an explicit update formula.

\begin{theorem}[Primal-dual update rule]\label{thm:update-rule}
	Let $\rho_k \in \mathbullet{\mathcal{A}}$ and $\alpha > 0$. Suppose $\varphi$ is the Fermi--Dirac entropy and 
  \Cref{lem:gradF} holds such that the primal representation $\nabla F\colon \mathcal{A} \to L^\infty(\Omega)$ exists.
	Then $\rho_{k+1}$ in \eqref{eq:mirror-descent-minimizer} is uniquely determined.
	Furthermore, we have $\rho_{k+1} \in \mathbullet{\mathcal{A}}$ and
	\begin{subequations}\label{eq:mirror-update-rule}
		\begin{align}
			\label{eq:dual-update}
			\sig^{-1}(\rho_{k+1})=\sig^{-1}(\rho_k)-\alpha\nabla F(\rho_k)-\mu_{k+1},
			\intertext{where, $\mu_{k+1}\in \mathbb{R}$ solves the volume correction equation,}
			\label{eq:volume-correction}
			\int_\Omega \sig(\sig^{-1}(\rho_k)-\alpha \nabla F(\rho_k)-\mu_{k+1})\dd x = \theta|\Omega|.
		\end{align}
	\end{subequations}
\end{theorem}

\begin{proof}
	The existence and uniqueness of $\rho_{k+1} \in \mathcal{A}$ is immediate from \Cref{thm:entropy-projection} and $\nabla F(\rho_k)\in L^\infty(\Omega)$.

It remains to prove that $\rho_{k+1}\in \mathbullet{\mathcal{A}} \subsetneq \mathcal{A}$; \cref{eq:mirror-update-rule} follows from the corresponding first-order optimality conditions.

\noindent{\textsl{Step 1}}. We claim that $\rho_{k+1}\in \mathbullet{\mathcal{A}}$.
	We use standard optimization theory and Lagrangian duality to prove this claim, e.g., \cite[Chap. 3.1]{Bonnans2000}. 

Recall that $\mathcal{A} = L^{\infty}_{[0,1]}(\Omega) \cap L^{\infty}_{\theta}(\Omega)$.
	In the current setting, Robinson's constraint qualification (RCQ) amounts to proving that the image of the mapping 
	$G(\rho) := (\rho,1) - \theta|\Omega|$ for $\rho \in L^{\infty}_{[0,1]}(\Omega)$ contains an open interval 
	around zero. Since the minimum value in the range is $G(0) = - \theta|\Omega| < 0$,
	the maximum value in the range is $G(1) = (1-\theta) |\Omega| > 0$, and $G$ is continuous we conclude that
	RCQ holds. Letting $\mathcal{L}(\rho,\mu) := \alpha J_k(\rho) - \mu( (\rho,1) - \theta |\Omega|)$, we deduce that 
	there exists a Lagrange multiplier  $\mu_{k+1}\in \mathbb{R}$ such that 
	\begin{equation}
		\label{eq:SaddlePointMinimizer}
		\mathcal{L}(\rho_{k+1},\mu_{k+1})\leq \mathcal{L}(\rho,\mu_{k+1})~\text{ for all } \rho\in L^{\infty}_{[0,1]}(\Omega).
	\end{equation}

We are left to suppose that $\rho_{k+1}\not\in \operatorname{int} L^\infty_{[0,1]}(\Omega)$.
	By setting $\rho=\rho_\epsilon:=\max\{\epsilon,\min\{1-\epsilon,\rho_{k+1}\}\}\in \operatorname{int} L^\infty_{[0,1]}(\Omega)$ with $0<\epsilon<1/e$, we have from the definition of $\mathcal{L}$ and $\rho_\epsilon$ and \Cref{lem:fermi-dirac} that
	\begin{equation}\label{eq:interior-aux}
		0\leq \int_{\Omega_\epsilon}\Big(\alpha\nabla F(\rho_k) + \mu_{k+1}-\sig^{-1}(\rho_k)\Big)(\rho_\epsilon-\rho_{k+1})\dd x + \varphi(\rho_\epsilon)-\varphi(\rho_{k+1})
		.
	\end{equation}
	Here, $\Omega_\epsilon$ is the union of $\Omega_{\epsilon-}:=\{x\in\Omega:\rho_{k+1}(x)<\epsilon\}$ and $\Omega_{\epsilon+}:=\{x\in\Omega:\rho_{k+1}(x)>1-\epsilon\}$.
	We now rewrite the last two terms on the right-hand side as follows:
	\begin{align*}
		\begin{aligned}
			\varphi(\rho_\epsilon)-\varphi(\rho_{k+1}) & =\int_{\Omega_\epsilon}\Big(\epsilon\ln \epsilon-\rho_{k+1}\ln \rho_{k+1}\Big)\dd x                       \\
			                                           & \quad +\int_{\Omega_\epsilon}\Big((1-\epsilon)\ln(1-\epsilon)-(1-\rho_{k+1})\ln (1-\rho_{k+1})\Big)\dd x.
		\end{aligned}
	\end{align*}
	Let us observe that $x\ln x>(1-x)\ln(1-x)>\epsilon\ln\epsilon$ for all $x\in (1-\epsilon,1)$ when $0<\epsilon<1/e$.
	This shows that the integrand of the first term on the right-hand side is negative on $\Omega_{\epsilon+}$.
	Noticing the symmetricity, we can also show that the integrand of the second term on the right-hand side is negative on $\Omega_{\epsilon-}$.
	These give us the upper bound
	\begin{equation}\label{eq:entropy-diff}
		\begin{aligned}
			\varphi(\rho_\epsilon)-\varphi(\rho_{k+1}) & \leq\int_{\Omega_{\epsilon-}}\Big(\epsilon\ln \epsilon-\rho_{k+1}\ln \rho_{k+1}\Big)\dd x                    \\
			                                           & \quad +\int_{\Omega_{\epsilon+}}\Big((1-\epsilon)\ln(1-\epsilon)-(1-\rho_{k+1})\ln (1-\rho_{k+1})\Big)\dd x.
		\end{aligned}
	\end{equation}

We now focus on the first term on the right-hand side of~\cref{eq:entropy-diff}.
	Applying the mean value theorem and observing that $(y\ln y)' = \ln y + 1$ is an increasing function, we obtain
	\begin{multline*}
		\int_{\Omega_{\epsilon-}}\Big(\epsilon\ln\epsilon - \rho_{k+1}\ln\rho_{k+1}\Big)
		\\
		\leq \int_{\Omega_{\epsilon-}}(\epsilon-\rho_{k+1})(\ln \epsilon + 1)\dd x
		=(\ln\epsilon + 1)\|\rho_{\epsilon}-\rho_{k+1}\|_{L^1(\Omega_{\epsilon-})}.
	\end{multline*}
	The last term in \eqref{eq:entropy-diff} can be bounded similarly, and hence \eqref{eq:entropy-diff} becomes
	\begin{align*}
		\varphi(\rho_\epsilon)-\varphi(\rho_{k+1})\leq (\ln \epsilon + 1)\|\rho_\epsilon - \rho_{k+1}\|_{L^1(\Omega_\epsilon)}.
	\end{align*}

This and H\"{o}lder's inequality give the following upper bound on \cref{eq:interior-aux}:
	\begin{equation}
	\label{eq:epsilon_bounds}
		0\leq \Big(\|\alpha\nabla F(\rho_k) + \mu_{k+1}-\sig^{-1}(\rho_k)\|_{L^\infty(\Omega)}+\ln\epsilon + 1\Big)\|\rho_\epsilon-\rho_{k+1}\|_{L^1(\Omega_{\epsilon})}.
	\end{equation}
	Clearly, the right-hand side becomes negative for sufficiently small $\epsilon$ unless $\Omega_{\epsilon}$ has zero measure, which is forbidden by the assumption $\rho_{k+1}\not\in \operatorname{int} L^\infty_{[0,1]}(\Omega)$.
	We therefore conclude that $\rho_{k+1} \in \operatorname{int} L^\infty_{[0,1]}(\Omega)$.

\noindent{\textsl{Step 2}}. We now derive the update formulae~\cref{eq:mirror-update-rule}.
	\Cref{lem:fermi-dirac} tells us that $D_\varphi(\cdot,\rho_k)$ is differentiable at $\rho_{k+1}$ because $\rho_{k+1}\in \operatorname{int} L^\infty_{[0,1]}(\Omega)$.
	We may now write out the first optimality condition for~\cref{eq:SaddlePointMinimizer}: namely,
	\begin{equation}\label{eq:first-order-optimality}
		\int_\Omega \big(\alpha \nabla F(\rho_k) + \mu_{k+1} + \sig^{-1}(\rho_{k+1})-\sig^{-1}(\rho_k)\big)(q-\rho_{k+1})\dd x\geq 0
	\end{equation}
	for all $q\in L^{\infty}_{[0,1]}(\Omega)$.
	Using again that $\rho_{k+1}\in \operatorname{int}L^\infty_{[0,1]}(\Omega)$, there exists an $L^\infty(\Omega)$-open ball of some radius $\delta > 0$ such that $B(\rho_{k+1};\delta)\subset L^{\infty}_{[0,1]}(\Omega)$.
	Taking any point $q\in B(0;\delta)$ and testing with $\rho = \rho_{k+1} \pm q$ in~\cref{eq:first-order-optimality}, we deduce that
	\begin{equation}
		\label{eq:VariationalEquality}
		\int_\Omega \big(\alpha \nabla F(\rho_k) + \mu_{k+1} + \sig^{-1}(\rho_{k+1})-\sig^{-1}(\rho_k)\big)q\dd x = 0
	\end{equation}
	for all $q\in B(0;\delta)$.

In fact,~\cref{eq:VariationalEquality} holds for all $q\in L^\infty(\Omega)$ because $B(0;\delta)$ is an absorbing set.

	We now use the fact that $C^\infty_{\mathrm{c}}(\Omega) \subset L^\infty(\Omega)$ and the vanishing integral theorem to conclude that
	\[
		\alpha \nabla F(\rho_k) + \mu_{k+1} + \sig^{-1}(\rho_{k+1})-\sig^{-1}(\rho_k) = 0
		~~
		\text{a.e.\ in } \Omega.
	\]
	Rearranging terms and using the fact that $\sig^{-1} \colon (0,1)\rightarrow\mathbb{R}$ is strictly monotonic and surjective, 
	we can thus express $\rho_{k+1}\in \mathbullet{\mathcal{A}}$ as follows:
	\begin{equation}\label{eq:update-rho}
		\rho_{k+1}=\sig\big(\sig^{-1}(\rho_k)-\alpha\nabla F(\rho_k)-\mu_{k+1}\big).
	\end{equation}

The proof is completed by noting that the volume constraint $\int_\Omega \rho_{k+1}\dd x=\theta|\Omega|$ is now equivalent to the volume correction equation \cref{eq:volume-correction}, which determines the unique Lagrange multiplier in
	\cref{eq:SaddlePointMinimizer}.
\end{proof}

It is a straightforward exercise to verify that the volume correction~\cref{eq:volume-correction} is equivalent to finding the Bregman projection $\mathcal{P}_\varphi:\operatorname{int}L^\infty_{[0,1]}(\Omega)\rightarrow {\mathcal{A}}$ such that
\begin{align*}
	\mathcal{P}_\varphi(q)=\argmin_{\rho\in{\mathcal{A}}}D_\varphi(\rho,q).
\end{align*}
This allows us to rewrite \eqref{eq:mirror-update-rule} as
\begin{subequations}\label{eq:pmd}
	\begin{align}
		\rho_{k+1/2} & =\sigma(\sigma^{-1}(\rho_k)-\alpha\nabla F(\rho_k)), \\
		\rho_{k+1}   & =\mathcal{P}_\varphi(\rho_{k+1/2}),
	\end{align}
\end{subequations}
revealing the key structural similarity with the projected gradient method \eqref{eq:pgd}.

We note that the projection $\mathcal{P}_\varphi$ is a smooth operator in contrast to $\mathcal{P}_{L^2}$ in~\cref{eq:L2proj}.

\begin{remark}\label{rem:eps-bound}
Note that~\cref{eq:epsilon_bounds} in the proof of \cref{thm:update-rule} gives a rather concrete bound on the distance $\epsilon=\epsilon(k)$ to the bounds at the $(k+1)$-th iteration:
\[
0 < \epsilon \le \exp(-1 - \|\alpha\nabla F(\rho_k)+\mu_{k+1}-\sig^{-1}(\rho_k)\|_{L^\infty(\Omega)}).
\]
After introducing the latent iterates $\psi_k$ below, we can write this equivalently as 
$
0 < \epsilon \le \exp(-1 - \|\psi_{k+1}\|_{L^\infty(\Omega)}).
$
As we show below, $\psi_{k}$ becomes unbounded on the active sets as $k \to +\infty$, which indicates that $\rho_k$ will tend rapidly to the
bounds on the active sets.
\end{remark}

\subsection{SiMPL: A latent variable update rule}
\label{sub:LVMD}
Ideally, the optimal designs $\rho^{\star}$ that solve \eqref{eq:red-top-opt} will be binary over $\Omega$. At the very least, we expect there will exist subsets of 
positive Lebesgue measure on which $\rho^\star=0$ or $1$. This presents a problem for the primal-dual update rule~\cref{eq:pmd} as the discretized 
version of \eqref{eq:pmd} may suffer from two types of numerical instability: $\sigma^{-1}(x) = \ln(x) - \ln(1-x)$ for $x$ near $0$ and $1$, and  oscillations
due to the method of discretization.

In \cite{keith2023proximal}, the authors introduced a latent variable $\psi=\varphi^\prime(\rho)=\sig^{-1}(\rho)$ to remedy these issues. This leads to an update 
rule for the latent variables of the form:
\begin{subequations}\label{eq:latent-update}
	\begin{align}
	\label{eq:latent-update-eq1}
		\psi_{k+1/2} & =\psi_k-\alpha_k\nabla F(\rho_k), \\
	\label{eq:latent-update-eq2}
		\psi_{k+1}   & =\psi_{k+1/2}-\mu_{k+1},\\
    \intertext{where $\mu_{k+1}$ solves}
	\label{eq:latent-update-eq3}
    \int_\Omega \sig(\psi_{k+1/2}&-\mu_{k+1})\dd x=\theta|\Omega|.
	\end{align}
\end{subequations}
Using the latent variable $\psi$ effectively removes the unstable logarithmic term from the algorithm.
Moreover, we are free to use both high-order polynomial bases to discretize the associated function space for $\psi_k$. Whenever needed, we can always transform the discrete approximation 
$\psi_{k,h}$ into a density $\rho_{k,h} := \sigma(\psi_{k,h})$. This has significant advantages as illustrated in \Cref{fig:gibbs}. Here, we see that while the latent variable $\psi_{k,h}$ still exhibits the 
typical Runge phenomenon when attempting to approximate a binary function with a high-degree polynomial, such oscillations become imperceptible under the sigmoidal transformation
because $\sigma(x)$ approaches to $0$ or $1$ exponentially as $x\rightarrow \pm\infty$.
\begin{figure}
	\centering
  \includegraphics[width=0.37\textwidth]{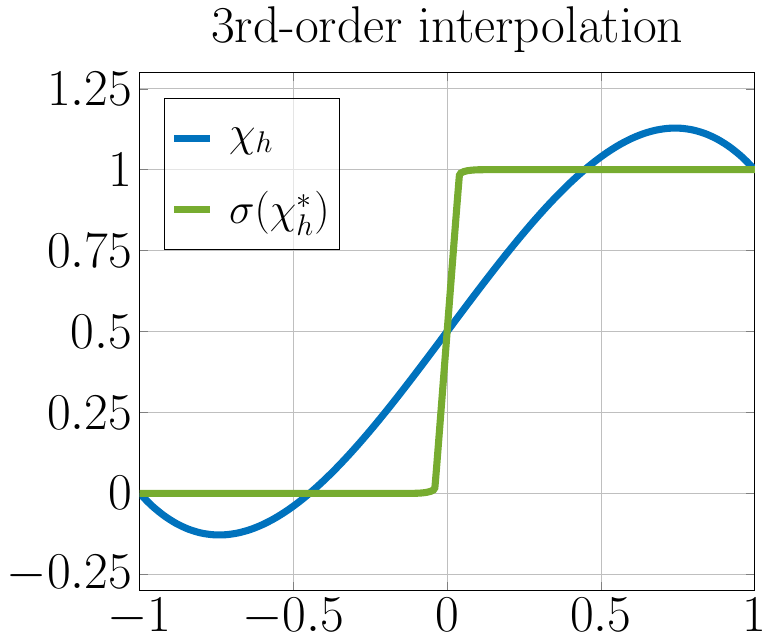}
  \includegraphics[width=0.37\textwidth]{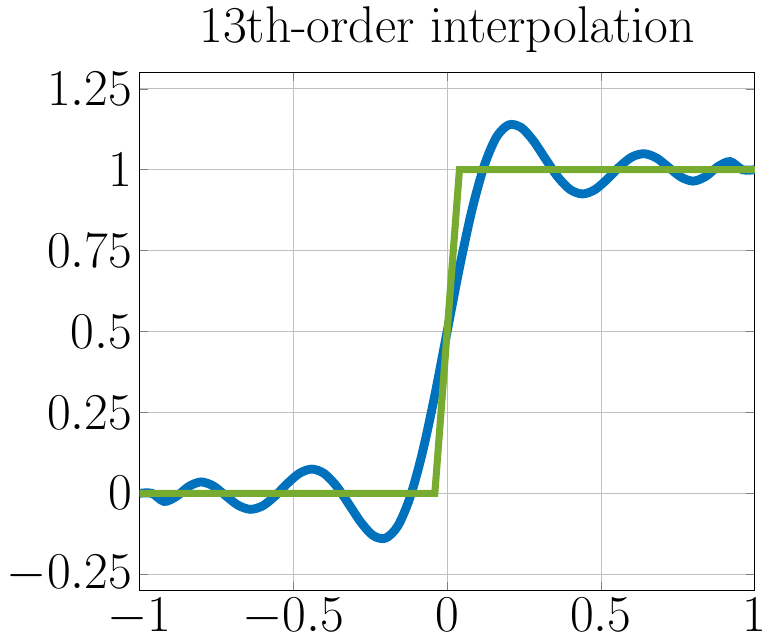}
	\caption{
		Interpolation of a characteristic function $\chi_{\{x>0\}}$ with 3rd (left) and 13th (right) order polynomials.
		$\chi_h$ is the direct interpolation, and $\chi^*_h$ is the interpolation in the latent space.
		We set $\chi^*_{h}= \min\{40, \max\{-40, \sig^{-1}(\chi_{\{x>0\}})\}\}$ to have a finite $\chi^*_{h}$.
		The $\sig(\chi_{h})$ used in SiMPL, shows a sharp interface without oscillation.
		The cut-off interpolation here is used only for demonstration, not in the optimization process.
		\label{fig:gibbs}}
\end{figure}

\begin{remark}\label{rem:mu-solve}
	The volume projection~\cref{eq:latent-update-eq3} is equivalent to solving the scalar nonlinear equation \eqref{eq:volume-correction}.
  We advocate for using a scalar root-finding method, such as Brent's method \cite{Brent1971} or the Illinois algorithm (a modified \textit{regula falsi} method, \cite{Dowell1971}), for maximum robustness when solving this equation.
	\Cref{lem:LM_bounds}, given in the following section, provides finite and computable upper and lower bounds that can be used to specify the corresponding search interval for $\mu_{k+1}$.

\end{remark}

\section{Preliminary convergence theory}\label{sec:prelim-conv}
We can draw a number of interesting conclusions about the convergence behavior of SiMPL before considering the step size $\alpha_k>0$ at each iteration.
We gather these observations here. The subsequent sections of the paper will introduce step size strategies and globalization schemes leading to the algorithms SiMPL-A and SiMPL-B. We pose the following assumption for convenience.
\begin{assumption}\label{as:seq-gen}
	$\widehat{F} : H^1(\Omega) \times L^{\infty}(\Omega) \to \mathbb R$ is Fr\'echet differentiable and the reduced functional $F$ has a Fr\'echet derivative that admits a primal representation $\nabla F$ such that  $\nabla F\colon \mathcal{A} \to L^\infty(\Omega)$.
\end{assumption}

\subsection{Convergence of the iterates} 
\label{sub:convergence_of_the_iterates}

Given a sequence of positive scalars $\{\alpha_k\}$, the sequence of densities $\left\{\rho_k\right\}$ and Lagrange multipliers $\left\{\mu_k\right\}$ is generated by the primal update rule \eqref{eq:mirror-update-rule}.
In what follows $\chi_{\mathcal{B}}$ is the characteristic function of a measurable set $\mathcal{B} \subset \Omega$.
\begin{proposition}[Convergence of $\rho_k$]\label{prop:prelim-rho-conv}
	Let \Cref{as:seq-gen} hold.
	Then $\left\{\rho_k\right\}$ admits an index set $K_{\rm w} \subset \mathbb{N}$ and weakly convergent subsequence $\left\{\rho_{k}\right\}_{k \in K_{\rm w}}$ in $L^p(\Omega)$ for $p \in [1, \infty)$ with limit point $\rho^{\star} \in \mathcal{A}$.
	Moreover, for the active sets,
	\[
		\Omega_{1} := \left\{ x \in \Omega \left| \; \rho^{\star}(x) = 1 \right.\right\}~ \text{ and }~
		\Omega_{0} := \left\{x \in \Omega \left| \; \rho^{\star}(x) = 0 \right.\right\},
	\]
	the sequences $\left\{(1 - \rho_{k})\chi_{\Omega_1}\right\}_{k \in K_{\rm w}}$ and  $\left\{\rho_{k}\chi_{\Omega_0}\right\}_{k \in K_{\rm w}}$ converge in measure to zero.
	Finally, there exists an index set $K_{\rm pn} \subset K_{\rm w}$ such that $\left\{(1 - \rho_{k})\chi_{\Omega_1}\right\}_{k \in K_{\rm pn}}$ and  $\left\{ \rho_{k} \chi_{\Omega_0}\right\}_{k \in K_{\rm pn}}$ converge to zero both pointwise and strongly in $L^p(\Omega)$, for every $p \in [1,\infty)$.
\end{proposition}
\begin{proof}
	By construction, $\left\{\rho_k\right\} \subset \mathcal{A}$. More specifically, $0 \le \rho_k\le 1$ and $\int_{\Omega} \rho_k \ \mathrm{d}x = \theta |\Omega|$ for all $k \in \mathbb N$. Hence, $\left\{\rho_k\right\}$ is bounded in every $L^p$-space $p \in [1,\infty]$.
	Consequently, there exists a weakly convergent subsequence $\left\{\rho_{k_l}\right\}$ in $L^p(\Omega)$ for any $1 < p < \infty$ with weak limit point $\rho^{\ast}$. For $p =1$ we appeal to the Dunford-Pettis theorem as used in the proof of \Cref{thm:entropy-projection}. The weak limit point $\rho^{\ast}$ belongs to $\mathcal{A}$ since norm closed convex sets are also weakly closed. We denote the indices for the weakly convergent subsequence by $k \in K_{\rm w}$.

	According to \Cref{thm:update-rule}, $0 < \rho_{k} < 1$ for all $k$. Recall then from Markov's inequality that for any $\varepsilon > 0$ we have
	\[
		\left| \left\{x \in \Omega \left| |(1 - \rho_k)\chi_{\Omega_1}| \ge \varepsilon \right.\right\}\right| \le \frac{1}{\varepsilon} \int_{\Omega}|\chi_{\Omega_1}(1 - \rho_k)| \, \mathrm{d}x.
	\]
	Passing to $k \in K_{\rm w}$, the upper bound yields
	\[
		\frac{1}{\varepsilon} \int_{\Omega}|(1 - \rho_{k})\chi_{\Omega_1}| \, \mathrm{d}x = \frac{1}{\varepsilon} \int_{\Omega_1} (1 - \rho_{k}) \, \mathrm{d}x = \frac{1}{\epsilon}(|\Omega_1| - (\rho_{k},\chi_{\Omega_1})) \to 0.
	\]
	Hence, $(1 - \rho_{k})\chi_{\Omega_1} \to 0$ in measure (along the subsequence defined by $k \in K_{\rm w}$).  The same argument holds by symmetry for $\rho_k\chi_{\Omega_0} $ with $k \in K_{\rm w}$.

	Finally, the remaining assertions follow from the fact that the Lebesgue measure is $\sigma$-finite, which implies the existence of a subset of indices $K_{\rm pn} \subset K_{\rm w}$ such that  $\left\{(1 - \rho_{k})\chi_{\Omega_1}\right\}_{k \in K_{\rm pn}}$ and  $\left\{ \rho_{k}\chi_{\Omega_0}\right\}_{k \in K_{\rm pn}}$ converge pointwise to zero, whereas the norm convergence follows from  the dominated convergence theorem due to the pointwise uniform bounds.
\end{proof}
\Cref{prop:prelim-rho-conv} provides us with some deeper intuition about the convergence behavior of the iterates, even without an update strategy for $\alpha_k$. The following corollary provides a stronger convergence statement provided $\rho^{\star}$ is binary-valued (as observed in our numerical experiments).

\begin{corollary}[Convergence of $\rho_k$ for binary-valued $\rho^{\star}$]
  Let \Cref{as:seq-gen} hold. Suppose $\Omega = \Omega_{1} \cup \Omega_{0}$ holds up to a set of measure zero, i.e., $\rho^{\star}$ is a binary design. Then $\{\rho_{k}\}_{k \in K_{\rm pn}}$ strongly converges in $L^p(\Omega)$ to $\rho^{\star}$ for any $p \in [1,\infty)$.
\end{corollary}
\begin{proof}
	This follows from \Cref{prop:prelim-rho-conv} in light of the fact that $|\rho_{k} - \rho^{\star}| \le |\chi_{\Omega_1}(\rho_{k} - 1)| + |\chi_{\Omega_0}\rho_k|$ holds almost everywhere on $\Omega$.
\end{proof}

Similar to the two previous results, we can prove some intuitive properties about the convergence behavior of the latent variables.
The latent variables do not live in the dual space to $\rho$, in contrast to the original mirror descent algorithm \cite{nemirovskij1983problem}.
In fact, the natural setting for $\psi$ is a metric space of extended real-valued functions whose metric is induced by $\sigma$ and an $L^p$-norm under the convention that $\sigma(+\infty) = 1$ and $\sigma(-\infty) = 0$.
In particular, the distance between two latent variables $\psi, \varphi$ can be measured by the metrics $d_{p}(\psi,\varphi) := \| \sigma(\psi) - \sigma(\varphi) \|_{L^p(\Omega)}$, for any $p \in [1,\infty]$.

\begin{proposition}[Convergence of $\psi_k$]\label{prop:psi-con}
	Let \Cref{as:seq-gen} hold and let $\left\{\psi_k\right\}$ be the sequence of latent variables associated to $\left\{\rho_k\right\}$. Define $\psi^\star := \sigma^{-1}(\rho^{\star})$, where
	$\psi^{\star} = +\infty$ where $\rho^{\star} = 1$ and $\psi^{\star} = -\infty$ where $\rho^{\star} = 0$.
	Then the sequence $\left\{\bar{\psi}_k\right\}_{k \in K_{\rm pn}}$ with
	\[
		\bar{\psi}_k(x) :=
		\begin{cases}
			\psi_k(x)       & \text{ if } x \in \Omega_1 \cup \Omega_0, \\
			\psi^{\star}(x) & \text{ else,}
		\end{cases}
	\]
	tends to $\psi^{\star}$ in the metric $d_{p}$ for any $p \in [1,\infty)$. If $\rho^{\star}$ is binary-valued, then $\left\{\psi_{k}\right\}_{k \in K_{\rm pn}}$ converges in $d_{p}$ to $\psi^{\star}$.
\end{proposition}
\begin{proof}
	The assertion is a simple consequence of \Cref{prop:prelim-rho-conv} given that
	for almost every $x \in \Omega$, any $p \in [1,\infty)$, and any $k \in K_{\rm pn}$, we have
	\[
		|\sigma(\psi_k) - \sigma(\psi^{\star})|^p =
		|(\sigma(\psi_k) - \sigma(\psi^{\star}))\chi_{\Omega_1 \cup \Omega_0}|^p =
		|(\rho_{k} - \rho^{\star})\chi_{\Omega_1\cup\Omega_0}|^p.
	\]
\end{proof}

\begin{remark}
	To understand the space which endowed with the metric $d_p$, we let $\Psi$ be the set of all extended real-valued functions on $\Omega$ such that $\psi \in \Psi$ implies $\sigma(\psi) \in L^p(\Omega)$.
	We define an equivalence class $\sim$ such that $f \sim g$ for $f,g \in \Psi$ if $\sigma(f) = \sigma(g)$ almost everywhere on $\Omega$.
	The latent variable space is then the quotient space $\Psi' := \Psi/\sim$, which we endow with the metric topology given by $d_{p}$.
\end{remark}

\subsection{Towards a KKT system} 
\label{sub:towards_a_kkt_system}

We now seek to derive a preliminary KKT-type optimality system using some additional assumptions on the types of convergence that are largely guaranteed by the globalization strategies in the following section. This concerns both weak and pointwise convergence of the Lagrange multiplier estimators as well as the parameter $\alpha_k$.

In the sequel, we let
\begin{equation}
	\label{eq:lagrange}
	\lambda_{k+1}
	=
	(\psi_{k+1}-\psi_{k})/\alpha_k
\end{equation}
denote approximate Lagrange multipliers for the bound constraints.
We begin by showing that both the multipliers for the volume constraint $\int_\Omega \rho \dd x = \theta |\Omega|$ and bound constraints $0 \leq \rho \leq 1$ a.e.\ in $\Omega$, i.e., $\mu_k/\alpha_k$ and $\lambda_k$, respectively, are bounded in a similar way.

\begin{lemma}
	\label{lem:LM_bounds}
	Let \Cref{as:seq-gen} hold true.
	Then we have the following bounds:
	\begin{equation}
		\label{eq:LagrangeMultiplierApprox}
		|\mu_{k+1}|/\alpha_k \leq \|\nabla F(\rho_k)\|_{L^\infty(\Omega)},
		\qquad
		\|\lambda_{k+1}\|_{L^\infty(\Omega)} \leq 2\|\nabla F(\rho_k)\|_{L^\infty(\Omega)},
	\end{equation}
	for each $k \in \mathbb{N}$.
\end{lemma}
\begin{proof}
	The monotonicity of $\sigma$ implies that
	\begin{align*}
		\theta|\Omega|=\int_\Omega \sigma(\sigma^{-1}(\rho_k))\dd x \leq \int_\Omega \sigma(\sigma^{-1}(\rho_k)-\alpha_k\nabla F(\rho_k) + \alpha_k \|\nabla F(\rho_k)\|_{L^\infty(\Omega)})\dd x.
	\end{align*}
	Similarly, we have
	$
		\theta|\Omega|\geq \int_\Omega \sigma(\sigma^{-1}(\rho_k)-\alpha_k\nabla F(\rho_k)-\alpha_k\|\nabla F(\rho_k)\|_{L^\infty(\Omega)})\dd x.
	$
	Thus, the intermediate value theorem and~\cref{eq:volume-correction} implies that $|\mu_{k+1}|/\alpha_k\leq \|\nabla F(\rho_k)\|_{L^\infty(\Omega)}$.
	Finally, by~\cref{eq:latent-update}, we have
	\begin{align*}
		\|\lambda_{k+1}\|_{L^\infty(\Omega)}=\|-\nabla F(\rho_k)-\mu_{k+1}/\alpha_k\|_{L^\infty(\Omega)}\leq 2\|\nabla F(\rho_k)\|_{L^\infty(\Omega)}.
	\end{align*}

\end{proof}

Note that the estimates in~\Cref{lem:LM_bounds} hold for all $L^p(\Omega)$-norms up to constants depending only on $|\Omega| < \infty$ and $p \in [1,\infty]$.
Also note that $\rho_k$ enters into $\nabla F$ via the left-hand side of the filter equation~\cref{eq:filt-eq}.
Thus, it is reasonable to assume that $\nabla F$ is completely continuous from $\mathcal{A}$ into $L^2(\Omega)$ due to the compact embedding of $H^1(\Omega)$ into $L^2(\Omega)$ and \Cref{as:seq-gen}.

\begin{assumption}\label{as:complete_continuity}
	The primal representative $\nabla F$ from~\Cref{as:seq-gen} is completely continuous from $\mathcal{A}$ into $L^2(\Omega)$.
\end{assumption}

The assumption above is valid for numerous objective functions used in topology optimization, including the compliance objective utilized in \Cref{sec:num} below, when the filter provides sufficient compactness and regularity.
  Other possible objective functions are the average power dissipation in a fluid flow, softmax penalties ($L^p(\Omega)$-norm with large $p$) on the temperature for heat transfer, minimizing directional displacement at point for compliant mechanism design, and tracking type-functionals targeting a prescribed system response; see, e.g., \cite{Bendse2004, Elesin2012, Borrvall2003} and the references within.
We now employ these observations to prove the following proposition.

\begin{proposition}[Sufficient conditions for stationarity]
	\label{prop:KKT}
	Let \Cref{as:complete_continuity} hold true and let $K_{\rm w} \subset \mathbb{N}$ denote the index set in \Cref{prop:prelim-rho-conv}.
	Then there exist strong limits $\mu_{k+1}/\alpha_k \to \mu^{\star} \in \mathbb{R}$ and $\lambda_k \to \lambda^{\star} \in L^2(\Omega)$ as $k\to\infty$ within some index set $K_{\rm w}' \subset K_{\rm w}$.
	Moreover, these limits act as multipliers in the following stationarity equation:
	\begin{equation}\label{eq:KKT-1}
		\nabla F(\rho^{\star}) + \lambda^{\star} + \mu^{\star} = 0,
	\end{equation}
	where $\rho^{\star}$ is the weak limit point from \Cref{prop:prelim-rho-conv}.
\end{proposition}
\begin{proof}

	By~\Cref{lem:LM_bounds}, we know that there exists a constant $C > 0$ such that $\|\lambda_{k+1}\|_{L^2(\Omega)} \leq C\|\nabla F(\rho_k)\|_{L^2(\Omega)}$ for all $k \in K_{\rm w}$.
	Moreover, by the hypotheses, $\left\{\nabla F(\rho_{k})\right\}_{k \in K_{\rm w}}$ converges strongly in $L^2(\Omega)$ to $\nabla F(\rho^{\star})$.
	Thus, $\{\lambda_k\}_{k \in K_{\rm w}}$ is bounded in $L^2(\Omega)$ and we conclude that there exists a subsequence with index set $K_{\rm w}' \subset K_{\rm w}$ such that $\lambda_k \rightharpoonup \lambda^{\star} \in L^2(\Omega)$.
	Continuing, it follows from \eqref{eq:latent-update} that
	\begin{align*}
		(\mu_{k+1}/\alpha_k,1) = -(\nabla F(\rho_k),1) - (\lambda_{k+1},1)
		\,.
	\end{align*}
	Thus,
	\[
		\mu_{k+1}/\alpha_k \to \mu^{\star}
		:=
		-\frac{1}{|\Omega|} \int_\Omega \nabla F(\rho^{\star}) + \lambda^{\star} \dd x
		\,,
	\]
	in the limit of $K_{\rm w}' \ni k\to\infty$.
	Using this observation and complete continuity of $\nabla F(\rho_k)$, we deduce strong convergence of $\lambda_k$ in $L^2(\Omega)$ as $K_{\rm w}' \ni k\to\infty$.
	In particular, we find that
	\[
		\lambda_{k+1}
		=
		-\nabla F(\rho_k) - \mu_{k+1}/\alpha_k
		\to
		-\nabla F(\rho^\star) - \mu^\star
		\,.
	\]
	Since weak and strong limits must coincide, we conclude that $\lambda^\star = -\nabla F(\rho^\star) - \mu^\star$.

\end{proof}

The final result of this section uses mild conditions on the convergence behavior of SiMPL that in turn provide us with the missing complementarity conditions on $\lambda^{\star}$ in the previous result.

\begin{proposition}[Sufficient conditions for complementarity]
	\label{prop:compl}
	Define the approximating Lagrange multipliers by $\lambda_k := (\psi_{k+1}-\psi_{k})/\alpha_k$ for $k \in \mathbb N$. If $\lambda_k \rightharpoonup \lambda^{\star}$ in $L^1(\Omega)$, then for almost every $x \in \Omega$, we have
	\begin{subequations}
		\begin{equation}\label{eq:lambda-signs:eqs1and2}
			\lambda^{\star}(x)
			\begin{cases}
				\geq 0 & \text{if } \rho^{\star}(x) = 1, \\
				\leq 0 & \text{if } \rho^{\star}(x) = 0. 

			\end{cases}
		\end{equation}
		Moreover, if there exists a global lower bound $\underline \alpha \leq \alpha_k$ and a subsequence $\{\lambda_k\}_{k \in K_{\rm pn}}$ such that $\lambda_k \to \lambda^{\star}$ pointwise, then
		\begin{equation}\label{eq:lambda-signs:eq3}
			\lambda^{\star}(x) = 0 \quad\text{for almost every } x \in \Omega \text{ satisfying } 0 < \rho^{\star}(x) < 1
			\,.
		\end{equation}
	\end{subequations}
\end{proposition}

\begin{proof}
	\noindent\textsl{Step 1.}
	We begin by proving the second inequality in~\cref{eq:lambda-signs:eqs1and2}.
	Suppose that the set $S := \{ x \in \Omega_0 \mid \lambda^{\star}(x) > 0 \}$ has positive measure. Then,
	for any non-negligible subset $\mathcal{B} \subset S$, we have $\int_{\mathcal{B}} \lambda^{\star} \dd x > 0$ and for all sufficiently large $k$,
	$\int_{\mathcal{B}} \lambda_k \dd x > 0$, as well. Scaling the latter inequality by $\alpha_k$ yields a chain of inequalities for any given $\mathcal{B}$.
	In particular, there exists
	$k_0 \in \mathbb{N}$ such that
	\[ (\ell \ge k_0) ~\Rightarrow~
		\int_{\mathcal{B}} \psi_{\ell} \dd x \geq \int_{\mathcal{B}} \psi_{k_0} \dd x .
	\]
	Now, recall from~\cref{prop:prelim-rho-conv} that $\left\{\rho_\ell \right\}_{\ell \in K_{\rm pn}}$ converges pointwise a.e.\ on $\Omega_0$ to the constant function $0$.
	By Egorov's theorem, $\rho_{\ell} \to 0$ uniformly on measurable subsets $S_{\mu} \subset S$ such that $|S \setminus S_{\mu}| < \mu$ for any $\mu > 0$.
	In particular, for all $\epsilon \in (0,1)$, there exists $k_\epsilon > 0$ such that $\rho_\ell \leq \epsilon$ on $S_{\mu}$ for all $\ell \geq k_\epsilon$.
	Setting $\mathcal{B} = S_{\mu}$ for some fixed $\mu > 0$ and recalling that $\psi_\ell(x) = \sigma^{-1}(\rho_\ell(x))$ and $\sigma^{-1} : (0,1) \to \mathbb R$ is strict monotone
	increasing, we conclude that
	\[
		\sigma^{-1}(\epsilon)|S_\mu|
		\geq
		\int_{S_\mu} \psi_{\ell} \dd x
		\geq
		\int_{S_\mu} \psi_{k_0} \dd x,
	\]
	for all $\ell \in K_{\rm pn}$ greater than $\max\{k_0,k_\epsilon\}$.
	Taking any $\epsilon < \sigma(\int_{S_\mu} \psi_{k_0} \dd x / |S_\mu|)$, we deduce that $\int_{S_\mu} \psi_{k_0} \dd x > \int_{S_\mu} \psi_{k_0} \dd x$, which is absurd.
	Thus, $S$ cannot have positive measure.
	An analogous argument shows that $|\{ x \in \Omega_1 \mid \lambda^{\star}(x) < 0 \}| = 0$.
	\smallskip

	\noindent\textsl{Step 2.}
	We now show~\cref{eq:lambda-signs:eq3}.
	For every $\epsilon > 0$, let $S_\epsilon = \{ x \in \Omega \mid \rho^{\star}(x) \geq 2\epsilon,\ \lambda^{\star}(x) < 0 \}$ and assume that this set has positive measure for some $\epsilon > 0$.
	Therefore, for any measurable set $\mathcal{B} \subset S_\epsilon$, there exist constants $\delta > 0$ and $k_0 \in \mathbb{N}$ such that
	\[
		\int_\mathcal{B} \frac{\psi_{\ell+1}-\psi_\ell}{\alpha_\ell} \dd x \leq -\delta
		~~
		\text{for all }\ell \geq k_0.
	\]
	In particular, it holds that
	\[
		(\ell \ge k_0) ~\Rightarrow~
		\int_\mathcal{B} \psi_\ell \dd x \leq \int_\mathcal{B} \psi_{k_0} \dd x - \underline{\alpha} \delta (\ell - k_0)
		\,.
	\]
	By Egorov's theorem, we conclude that $\rho_{\ell} \to \rho^{\star}$ uniformly on measurable subsets $S_{\mu} \subset S_\epsilon$ such that $|S_\epsilon \setminus S_{\mu}| < \mu$ for any $\mu > 0$.
	In particular, there exists $k_\epsilon > 0$ such that $\rho_\ell > \epsilon$ on $S_{\mu}$ for all $\ell \geq k_\epsilon$.
	Setting $\mathcal{B} = S_{\mu}$ for some fixed $\mu > 0$ and again recalling that $\psi_\ell(x) = \sigma^{-1}(\rho_\ell(x))$ depends monotonically on $\rho_\ell(x)$, we find that
	\[
		\sigma^{-1}(\epsilon)|S_\mu|
		\leq
		\int_\mathcal{S_\mu} \psi_\ell \dd x \leq \int_\mathcal{S_\mu} \psi_{k_0} \dd x - \underline{\alpha} \delta (\ell - k_0)
		\,,
	\]
	for all $\ell \in K_{\rm pn}$ greater than $\max\{k_0,k_\epsilon\}$.
	Since the left-hand side of these inequalities is a constant and the right-hand side decreases linearly with $\ell$, there exists some $\ell > 0$ leading to a contradiction.
	Thus, we conclude that $|S_\epsilon| = 0$ for every $\epsilon > 0$.

	By the dominated convergence theorem, we can now show that $|\{ x \in \Omega\setminus\Omega_0 \mid \lambda^{\star}(x) < 0 \}| =
		\lim_{n} |\{ x \in \Omega \mid \rho^{\star}(x) \geq 1/n,\ \lambda^{\star}(x) < 0 \}| = 0$.
	An analogous argument shows that $|\{ x \in \Omega\setminus\Omega_1 \mid \lambda^{\star}(x) > 0 \}| = 0$.

	Together, both arguments imply that the set
	\begin{multline*}
		\{ x \in \Omega\setminus(\Omega_0\cup\Omega_1) \mid \lambda^{\star}(x) \neq 0 \}
		\\=
		(\{ x \in \Omega\setminus\Omega_1 \mid \lambda^{\star}(x) > 0 \}\setminus\Omega_0)
		\cup
		(\{ x \in \Omega\setminus\Omega_0 \mid \lambda^{\star}(x) < 0 \}\setminus\Omega_1)
	\end{multline*}
	has zero measure, as necessary.
\end{proof}

\begin{remark}
The main weakness of the previous statement is the assumption that the entire sequence of Lagrange multipliers converges weakly in $L^1(\Omega)$. In the next section, we will show that the line search strategies produce monotonically decreasing sequences of functional values along the iterates. However, it remains unclear how to leverage this fact to establish the convergence of the Lagrange multipliers.
\end{remark}
\section{Globalizing SiMPL}\label{sec:adaptive}

The convergence theory of the previous section lays a theoretical foundation. However, it says nothing about how to pick $\alpha_k > 0$ at each iteration nor whether we actually have descent in the objective function values along the sequence $\left\{\rho_k\right\}$. The goal of this section is to investigate the use of two line search strategies found in literature, and analyze their effect on SiMPL. This will lead to two different, globalized variants of SiMPL: SiMPL-A and SiMPL-B.

In both cases, we consider a standard backtracking update strategy. As is often the case, the choice of the initial step size is crucial for the performance of each method. This is made concrete below.
Given an initial step $\alpha_{k,0} > 0$, scalar $\beta \in (0,1)$ (typically $\beta = 1/2$), both methods choose a candidate $\alpha_{k,m} := \beta^m \alpha_{k,0}$ with $m = 0,1,2,\dots$ in the definition of $\rho_{k+1}$  (equivalently $\psi_{k+1}$) until
a sufficient decrease condition first holds after $m_k$ iterations. We then set $\alpha_{k} := \alpha_{k,m_k}$ and accept the iterate $\rho_{k+1}$ (equivalently $\psi_{k+1}$) and continue the algorithm.

Following the classical text by Bertsekas \cite[Eq. (22)]{bertsekas1976goldstein}, SiMPL-A will make use of a generalized Armijo backtracking line search: Given $0 < c_1 <  1$, set $m_k$ to be the smallest nonnegative integer for which
\begin{equation}\label{eq:armijo}
	F(\rho_{k+1}) \leq F(\rho_k) + c_1\int_\Omega \nabla F(\rho_k)(\rho_{k+1}-\rho_k)\dd x
\end{equation}
holds.
Another popular choice for sufficient decrease comes from \cite{zhou2019, mukkamala2020-cocain}: set $m_k$ to be the smallest nonnegative integer for which
\begin{equation}\label{eq:bregman-backtracking}
	\begin{aligned}
		 & F(\rho_{k+1}) \leq F(\rho_k) + \int_\Omega \nabla F(\rho_k)(\rho_{k+1} - \rho_k)\dd x + \frac{1}{\alpha_k}D_\varphi(\rho_{k+1},\rho_k) 
		\,.
	\end{aligned}
\end{equation}
The use of this line search strategy will be called SiMPL-B.

Note that in both cases, the Lagrange multiplier $\mu$, which acts as a volume correction term, is recomputed at each subiteration. We prove in~\Cref{lem:vol-shift-conti}, below, that $\mu$ is in fact continuously differentiable with respect to $\alpha$.

\begin{remark}
Both line search strategies are relatively easy to implement. SiMPL-A has a more straightforward interpretation as a backtracking line search and offers a bit of flexibility as one can easily adjust the parameter $c_1$. SiMPL-B is appealing as there are no parameters to tune and should perhaps be the default for new users. On the other hand, all line search rules are based on heuristics that suggest a step size based on the objective function value at the new iterate versus behavior of a local linear or nonlinear model function with the intention of providing sufficient decrease, so it makes sense to have several options.
\end{remark}

\subsection{Choosing the initial step sizes}
We suggest to choose initial step size $\alpha_{k,0}$ based on a kind of curvature condition similar to the well-known Barzilai--Borwein (BB) step \cite{BARZILAI1988}.
This is not an arbitrary choice, as it can be motivated by the convergence theory of the mirror descent algorithm. Following \cite{zhou2019}, we can rewrite the mirror descent
algorithm as a proximal point algorithm by introducing the functional $\hat{\varphi}_k=\alpha_k^{-1}\varphi - F$: Given $\rho_k$, $\alpha_k > 0$, find $\rho_{k+1}$ such that
\begin{align*}
	\rho_{k+1} = \argmin_{\rho\in \mathcal{A}} D_{\hat{\varphi}_k}(\rho,\rho_k)
\end{align*}
However, this requires $\hat{\varphi}_k$ to be differentiable and strictly convex, otherwise $D_{\hat{\varphi}_k}$ would not be a proper Bregman divergence.
This is, in fact, equivalent to $F$ being relatively smooth to $\varphi$ with constant $\alpha_k^{-1}$.
\begin{definition}[Relative smoothness]
	We say $F:\mathcal{A}\rightarrow \mathbb{R}$ is relatively smooth to $\varphi$ with constant $\gamma$ when
	\begin{equation}\label{eq:rel-smooth}
		\int_\Omega \Big(\nabla(\gamma^{-1}\varphi-F)(\rho)-\nabla(\gamma^{-1}\varphi-F)(q)\Big)(\rho-q)\dd x\geq 0\text{ for all } \rho,q\in \mathbullet{\mathcal{A}}.
	\end{equation}
\end{definition}
Using the strict monotonicity of $\nabla \varphi=\varphi'=\sig^{-1}$, we can rewrite \eqref{eq:rel-smooth} as follows:
\begin{align*}
	\gamma^{-1} \geq \sup_{\substack{\rho,q\in \mathbullet{\mathcal{A}} \\ \rho\neq q}}\frac{(\nabla F(\rho) - \nabla F(q), \rho-q)}{(\sig^{-1}(\rho) - \sig^{-1}(q), \rho - q)}.
\end{align*}
In general, the computation of $\gamma$ or the supremum here is intractable. However, taking $\rho=\rho_k,\;q=\rho_{k-1}$, we obtain a lower bound of $\gamma^{-1}$:
\begin{align*}
	\gamma^{-1}\geq \frac{(\nabla F(\rho_k) - \nabla F(\rho_{k-1}), \rho_k - \rho_{k-1})}{\int_\Omega(\psi_k - \psi_{k-1})(\rho_k - \rho_{k-1})\dd x}.
\end{align*}
As we will see below, $\gamma$ is related to the local Lipschitz continuity of $\nabla F$. Given the pivotal role played by the Lipschitz constant of $\nabla F$ in first-order methods
we reason that the following definition provides a good candidate for the initial step size:
\begin{equation}\label{eq:alpha-init-0}
	\tilde{\alpha}_{k,0} = \frac{(\psi_k - \psi_{k-1},\rho_k - \rho_{k-1})}{|(\nabla F(\rho_k) - \nabla F(\rho_{k-1}), \rho_k - \rho_{k-1})|}.
\end{equation}
We take the absolute value of the denominator to ensure $\tilde{\alpha}_{k,0} >0$. For the case where $\nabla F$ is monotone, $\tilde{\alpha}_{k,0}$ resembles the long Barzilai--Borwein step when one pair of
the successive iterates in the numerator is exchanged for $\psi_k - \psi_{k-1}$.
Noting that $\tilde{\alpha}_k$ only uses the latest two iterates to approximate a global constant $\gamma$.
We can further incorporate global information by utilizing the geometric mean
\begin{equation}\label{eq:alpha-init}
	\alpha_{k,0}        =\sqrt{\tilde{\alpha}_{k,0}\alpha_{k-1}}.
\end{equation}
Of course, $\alpha_{k,0} < \gamma$ is not guaranteed, which further implies the need for a line search.

\begin{remark}[Lipschitz continuity of $\nabla F$]
	If $\nabla F:L^1_{[0,1]}(\Omega)\rightarrow L^\infty(\Omega)$ is globally Lipschitz, then we can readily demonstrate relative smoothness.
	Indeed, denoting the Lipschitz constant by $L$, we have that
	\[
		(\nabla F(\rho)-\nabla F(q), \rho-q)
		\leq \|\nabla F(\rho)-\nabla F(q)\|_{L^\infty}\|\rho-q\|_{L^1}
		\leq L\|\rho-q\|_{L^1}^2
	\]
	for all $\rho,q\in \mathbullet{\mathcal{A}}$.
	Then the strong convexity of $\varphi$ in $L^1$, cf.\ \Cref{lem:fermi-dirac}, with constant $m$ implies that
	\begin{align*}
		(\nabla F(\rho)-\nabla F(q),\rho-q)\leq Lm^{-1} (\nabla \varphi(\rho)-\nabla \varphi(q),\rho-q).
	\end{align*}
	Therefore, $F$ is also relatively smooth to $\varphi$ with constant $\gamma = Lm^{-1}$.
\end{remark}

\subsection{Analysis of SiMPL-A}
We analyze the convergence behavior of SiMPL-A. In particular, we will see that the Armijo rule~\cref{eq:armijo} forces a strict monotonic decrease in the function values and convergence of the increments $\rho_{k+1} - \rho_{k} \to 0$ strongly in $L^2$ for the entire sequence.
Our analysis requires the following technical lemma, which we use to prove both \Cref{thm:gbb-finite-termination,thm:conv} below.
\begin{lemma}\label{lem:vol-shift-conti}
	Let $\rho\in \operatorname{int}L^\infty_{[0,1]}(\Omega)$ and $v\in L^\infty(\Omega)$.
	If we define $e: \mathbb R^2 \to \mathbb R$ by
	\begin{equation}
		e(\alpha,\mu):=\int_\Omega \sig(\sig^{-1}(\rho) - \alpha v - \mu)-\theta \dd x,
	\end{equation}
	then $\partial e/\partial \mu < 0$.
	Consequently, there exists a differentiable function $g:\mathbb{R}\rightarrow\mathbb{R}$ such that $g(0)=0$ and
	\begin{equation}\label{eq:mu}
		\int_\Omega \sig(\sig^{-1}(\rho)-\alpha v - g(\alpha))-\theta \dd x = 0.
	\end{equation}
	In other words, $\mu$ can be identified as a differentiable function of $\alpha$ with derivative $\mu'$ given by
	\begin{equation}\label{eq:mu-diff}
		\mu'(\alpha) = -\frac{(\sig'(\sig^{-1}(\rho)- \alpha v - \mu(\alpha),v)}{(\sig'(\sig^{-1}(\rho) - \alpha v - \mu(\alpha)),1)}
	\end{equation}
	where $\sig'(x) = \sig(x)(1-\sig(x))$ is the derivative of the sigmoid function.
\end{lemma}
\begin{proof}
	Since $\rho\in \operatorname{int}L^\infty_{[0,1]}(\Omega)$, we have $\sig^{-1}(\rho)\in L^\infty(\Omega)$.
	Also, $\sig$ is bounded, continuously differentiable, and $v\in L^\infty(\Omega)$.
	Therefore, $e(\alpha,\mu)$ is continuous and differentiable for each variable.
	We differentiate $e(\alpha,\mu)$ with respect to $\mu$ at $(\alpha,\mu)\in\mathbb{R}^2$, and observe
	\begin{align*}
		\begin{aligned}
			\frac{\partial e}{\partial \mu} & =
			\int_\Omega \sig(\sig^{-1}(\rho)-\alpha v+\mu)\Big(\sig(\sig^{-1}(\rho) - \alpha v + \mu) - 1\Big)\dd x < 0
		\end{aligned}
	\end{align*}
	Here, the last inequality holds since $\alpha,\mu$ are finite, $\rho\in \operatorname{int}L^\infty_{[0,1]}(\Omega)$, $v\in L^\infty(\Omega)$, and $0<\sig(q)<1$ almost everywhere when $q\in L^\infty(\Omega)$.
	The assertion then follows from the implicit function theorem, e.g., \cite[Theorem 1.3.1]{Krantz2003}.
\end{proof}

As usual, we need to first identify conditions under which the line search will terminate after a finite number of steps. This can be guaranteed under various global and local Lipschitz conditions on the reduced gradient $\nabla F$.
\begin{theorem}[Finite termination of the line search]\label{thm:gbb-finite-termination}
	Let $\rho_k,\rho_{k-1}\in\mathbullet{\mathcal{A}}$.
	If $F$ is relatively smooth to $\varphi$ with constant $\gamma$, then $\alpha_{k,0} \geq\gamma$.

	If $\nabla F:L^2_{[0,1]}(\Omega)\rightarrow L^2(\Omega)$ is locally Lipschitz continuous at $\rho_k$ with constant $L_k$, then the generalized Armijo line search \ref{eq:armijo} is guaranteed to terminate once
	\[
		\alpha_k\geq 2(1-c_1)/L_k.
	\]
	Further if $\nabla F$ is globally Lipschitz with constant $L$, then
	\begin{equation}\label{eq:alpha-bdd}
		\alpha_k \geq \min\{2(1-c_1)/L, \gamma\}=:\underline{\alpha}.
	\end{equation}
\end{theorem}
\begin{proof}
	First, we recall that
	\begin{align*}
		\alpha_{k,0} = \frac{\int_\Omega (\psi_k - \psi_{k-1})(\rho_k - \rho_{k-1})\dd x}{\int_\Omega\Big(\nabla F(\rho_k)-\nabla F(\rho_{k-1})\Big)(\rho_k - \rho_{k-1})\dd x}.
	\end{align*}
	The relative smoothness of $F$ and the monotonicity of $\sig^{-1}$ implies
	\begin{align*}
		\int_\Omega(\psi_k - \psi_{k-1})(\rho_k-\rho_{k-1})\dd x\geq \gamma\int_\Omega \Big(\nabla F(\rho_k) - \nabla F(\rho_{k-1})\Big)(\rho_k-\rho_{k-1})\dd x.
	\end{align*}
	Plugging this into the definition of $\alpha_{k,0}$ provides the lower bound.

	Define $\psi_k(\alpha):=\psi_k-\alpha \nabla F(\rho_k)-\mu(\alpha)$ and $\rho_k(\alpha):=\sig(\psi_k(\alpha))$, where $\mu$ is a differentiable function of $\alpha$ according to \Cref{lem:vol-shift-conti}.
	By the Mean Value Theorem, we have
	\begin{align*}
		F(\rho_k(\alpha)) - F(\rho_k) =\int_0^1 \int_\Omega \nabla F(\rho_k+t\delta\rho_k(\alpha))\delta\rho_k(\alpha)\dd x \dd t.
	\end{align*}
	where $\delta\rho_k(\alpha)=\rho_k(\alpha)-\rho_k$.
	By adding and subtracting $\int_\Omega \nabla F(\rho_k)\delta\rho_k(\alpha)\dd x$, we have
	\begin{align*}
		\begin{aligned}
			 & F(\rho_k(\alpha))-F(\rho_k)                                                                                                                                                     \\
			 & =\int_0^1\int_\Omega \Big(\nabla F(\rho_k+t\delta\rho_k(\alpha)) - \nabla F(\rho_k)\Big)\delta\rho_k(\alpha))\dd x \dd t+\int_\Omega \nabla F(\rho_k)\delta\rho_k(\alpha)\dd x.
		\end{aligned}
	\end{align*}
	The local Lipschitz continuity of $\nabla F$ implies
	\begin{align*}
		\begin{aligned}
			 & F(\rho_k(\alpha))-F(\rho_k)                                                                                             \\
			 & \leq L_k\int_0^1 t\|\delta\rho_k(\alpha))\|_{L^2(\Omega)}^2 \dd t+\int_\Omega \nabla F(\rho_k)\delta\rho_k(\alpha)\dd x \\
			 & =\frac{L_k}{2}\|\rho_k(\alpha) - \rho_k\|_{L^2(\Omega)}^2 + \int_\Omega \nabla F(\rho_k)(\rho_k(\alpha) - \rho_k)\dd x.
		\end{aligned}
	\end{align*}

By subtracting $c_1\int_\Omega \nabla F(\rho_k)(\rho_k(\alpha)-\rho_k)\dd x$ with $0<c_1<1$ from the both sides, we obtain
	\begin{equation}\label{eq:F_diff_aux}
		\begin{aligned}
			 & F(\rho_k(\alpha))-\Big(F(\rho_k)+c_1\int_\Omega \nabla F(\rho_k)(\rho_k(\alpha)-\rho_k)\dd x\Big)                            \\
			 & \leq \frac{L_k}{2}\|\rho_k(\alpha)-\rho_k\|_{L^2(\Omega)}^2+(1-c_1)\int_\Omega \nabla F(\rho_k)(\rho_k(\alpha)-\rho_k)\dd x.
		\end{aligned}
	\end{equation}
	Now we observe that the Mean Value Theorem yields
	\begin{equation}\label{eq:rho_diff}
		\begin{aligned}
			\rho_k(\alpha)-\rho_k
			 & =\alpha\int_0^1\rho_k'(t\alpha)\dd t                                                     \\
			 & =\alpha\int_0^1\rho_k(t\alpha)(1-\rho_k(t\alpha))(-\nabla F(\rho_k)-\mu'(t\alpha))\dd t.
		\end{aligned}
	\end{equation}
	Here, we utilized $\sig'(x)=\sig(x)(1-\sig(x))$.
	We now use the identity \eqref{eq:mu-diff} to arrive at
	\begin{equation}\label{eq:orthogonal}
		\begin{aligned}
			 & \int_\Omega \rho_k(t\alpha)(1-\rho_k(t\alpha))(-\nabla F(\rho_k)-\mu'(t\alpha))\dd x                                                                                                                            \\
			 & =\int_\Omega \rho_k(t\alpha)(1-\rho_k(t\alpha))\Big(-\nabla F(\rho_k)+\frac{\int_\Omega \rho_k(t\alpha)(1-\rho_k(t\alpha))\nabla F(\rho_k)\dd x}{\int_\Omega \rho_k(t\alpha)(1-\rho_k(t\alpha))\dd x}\Big)\dd x \\
			 & =0.
		\end{aligned}
	\end{equation}
	Therefore, the integrand in \eqref{eq:rho_diff} is orthogonal to all constants.

Noting that $\mu'(t\alpha)\in\mathbb{R}$, we can rewrite \eqref{eq:F_diff_aux} using the above two assertions by
	\begin{align*}
		 & F(\rho_k(\alpha))-\Big(F(\rho_k)+c_1\int_\Omega \nabla F(\rho_k)\big(\rho_k(\alpha)-\rho_k\big)\dd x\Big)                                      \\
		 & \leq \frac{L_k}{2}\|\rho_k(\alpha) - \rho_k\|_{L^2(\Omega)}^2                                                                                  \\
		 & \quad+\alpha(1-c_1)\int_\Omega\int_0^1 \nabla F(\rho_k)\Big(\rho_k(t\alpha)(1-\rho_k(t\alpha))(-\nabla F(\rho_k)-\mu'(t\alpha))\Big)\dd t\dd x \\
		 & = \frac{L_k}{2}\|\rho_k(\alpha)-\rho_k\|_{L^2(\Omega)}^2                                                                                       \\
		 & \;\;-\alpha(1-c_1)\int_0^1\int_\Omega\rho_k(t\alpha)(1-\rho_k(t\alpha))\Big(\nabla F(\rho_k)+\mu'(t\alpha)\Big)^2\dd x\dd t                    \\
		\begin{split}
			 & =\frac{L_k}{2}\|\rho_k(\alpha)-\rho_k\|_{L^2(\Omega)}^2                                                                              \\
			 & \quad -\alpha(1-c_1)\int_0^1\|\sqrt{\rho_k(t\alpha)(1-\rho_k(t\alpha))}(\nabla F(\rho_k)+\mu'(t\alpha))\|_{L^2(\Omega)}^2\dd x\dd t.
		\end{split}
	\end{align*}
	Here, we utilized Fubini's theorem and $0\leq \rho_k(t\alpha)(1-\rho_k(t\alpha))$.
	Since $0\leq \rho_k(t\alpha)(1-\rho_k(t\alpha))\leq 1/4$, we have
	\begin{align*}
		0\leq \rho_k(t\alpha)(1-\rho_k(t\alpha))\leq \sqrt{\rho_k(t\alpha)(1-\rho_k(t\alpha))}\leq \frac{1}{2}.
	\end{align*}
	This and the identity \eqref{eq:rho_diff} show that
	\begin{align*}
		 & F(\rho_k(\alpha))-\Big(F(\rho_k)+c_1\int_\Omega \nabla F(\rho_k)\big(\rho_k(\alpha)-\rho_k\big)\Big)                            \\
		 & \leq \frac{L_k}{2}\|\rho_k(\alpha)-\rho_k\|_{L^2(\Omega)}^2                                                                     \\
		 & \quad -\alpha(1-c_1)\int_0^1\|\rho_k(t\alpha)(1-\rho_k(t\alpha))(\nabla F(\rho_k)+\mu'(t\alpha))\|_{L^2(\Omega)}^2\dd t         \\
		 & \leq \frac{L_k}{2}\|\rho_k(\alpha)-\rho_k\|_{L^2(\Omega)}^2                                                                     \\
		 & \quad -\alpha(1-c_1)\Big\|\int_0^1\rho_k(t\alpha)(1-\rho_k(t\alpha))(\nabla F(\rho_k)+\mu'(t\alpha))\dd t\Big\|_{L^2(\Omega)}^2 \\
		 & \leq \left(\frac{L_k}{2}-\alpha(1-c_1)\right)\|\rho_k(\alpha)-\rho_k\|_{L^2(\Omega)}^2.
	\end{align*}
	Therefore, the back-tracking line search terminates when $\alpha < 2(1-c_1)/L_k$.
	Recalling that we halve $\alpha$ after each back-tracking line search iteration, the algorithm terminates with

$
		\alpha_k \geq (1-c_1)/L_k.
	$

	If $\nabla F$ is globally Lipschitz continuous with constant $L$, then we have

$
		\alpha_k \geq (1-c_1)/L_k\geq (1-c_1)/L.
	$

	This completes the proof.
\end{proof}

We emphasize here that the proof of finite termination for the Armijo condition uses the Lipschitz constant $L$, which is more restrictive than using the relative smoothness constant.
Furthermore, we highlight that we do not know whether the step size at termination will satisfy $\alpha_k < \gamma$.
Since the analysis in \cite{zhou2019} relies on the step sizes being smaller than the relative smoothness constant, we cannot simply carry over the results in \cite{zhou2019}.

\begin{theorem}\label{thm:conv}
	Assume that $\nabla F:L^2_{[0,1]}(\Omega)\rightarrow L^2(\Omega)$ is globally Lipschitz continuous with constant $L$, and let $\{\rho_k\}$ and $\left\{\alpha_k\right\}$ be generated by the SiMPL-A method.
	Then the following properties hold:
	\begin{enumerate}[label={(\roman*)}]
		\item $\{F(\rho_k)\}$ is strictly monotonically decreasing;
		      \label{item:conv_0}
		\item $\lim_{k\to\infty} \|\rho_k - \rho_{k+1}\|_{L^2(\Omega)} = 0$; and
		      \label{item:conv_1}
		\item $\liminf_{k\to\infty} k\|\rho_k - \rho_{k+1}\|_{L^2(\Omega)}^2 = 0$.
		      \label{item:conv_2}
	\end{enumerate}
	Moreover, if $F$ is completely continuous, then $F(\rho_k) \to F(\rho^\star)$, where $\rho^{\star}$ is the weak limit point of $\left\{\rho_{k}\right\}_{k \in K_{\rm w}}$.
\end{theorem}

\begin{proof}

	\textsl{Step 1.} We show \Cref{item:conv_0,item:conv_1}.

	We begin with the Armijo condition \eqref{eq:armijo}.

By plugging in the identity \eqref{eq:rho_diff}, we obtain
	\begin{multline*}
		F(\rho_{k+1})\leq \\ F(\rho_k) + c_1\alpha_k\int_0^1\int_\Omega \nabla F(\rho_k)\rho_k(t\alpha_k)(1-\rho_k(t\alpha_k))(-\nabla F(\rho_k)-\mu'(t\alpha_k)))\dd x\dd t.
	\end{multline*}
	where $\rho_k(t\alpha_k)$ and $\mu'(t\alpha_k)$ are defined as in \Cref{thm:gbb-finite-termination}.
	Similar to before, we use the orthogonality condition \eqref{eq:orthogonal} and boundedness of $\rho_k(\cdot)$ and $\alpha_k$ \eqref{eq:alpha-bdd} to obtain
	\begin{align*}
		\begin{aligned}
			F(\rho_{k+1})
			 & \leq F(\rho_k) - c_1\alpha_k\int_0^1\|\sqrt{\rho_k(t\alpha_k)(1-\rho_k(t\alpha_k))}(\nabla F(\rho_k)+\mu'(t\alpha_k))\|_{L^2(\Omega)}^2\dd t \\
		\end{aligned}
	\end{align*}
	The sum telescopes, leaving
	\begin{multline*}
		F(\rho_{n+1})\leq\\
		F(\rho_0) - \sum_{k=0}^nc_1\alpha_k\int_0^1\|\sqrt{\rho_k(t\alpha_k)(1-\rho_k(t\alpha_k))}(\nabla F(\rho_k)+\mu'(t\alpha_k))\|_{L^2(\Omega)}^2 \dd t.
	\end{multline*}
	Since \eqref{eq:red-top-opt} admits a solution and SiMPL-A generates a sequence of feasible points the function values are uniformly bounded from below. Therefore, the series converges and we have
	\begin{multline}\label{eq:sqrt_gradF_bdd}
		\alpha_k\int_\Omega \nabla F(\rho_k)(\rho_{k+1}-\rho_k)\dd x
		=\\\alpha_k\int_0^1\|\sqrt{\rho_k(t\alpha_k)(1-\rho_k(t\alpha_k))}(\nabla F(\rho_k)+\mu'(t\alpha_k))\|_{L^2(\Omega)}^2\dd t\rightarrow 0.
	\end{multline}
	Since the last term is non-negative, $\{F(\rho_k)\}$ is a strictly decreasing sequence.
	Again, similarly to \Cref{thm:gbb-finite-termination}, we further proceed by utilizing $0<\rho_k(t\alpha_k)(1-\rho_k(t\alpha_k))\leq 1/4$ and the identity \eqref{eq:rho_diff} to arrive at
	\begin{equation}
		\label{eq:Descent}
		F(\rho_{n+1})\leq F(\rho_0)-c_1\sum_{k=0}^n\alpha_k\|\rho_{k+1}-\rho_k\|_{L^2(\Omega)}^2.
	\end{equation}
	Therefore, $\underline{\alpha}\|\rho_k-\rho_{k+1}\|_{L^2(\Omega)}\leq \alpha_k\|\rho_k - \rho_{k+1}\|_{L^2(\Omega)}\rightarrow 0$.
	\smallskip

	\noindent\textsl{Step 2.}
	Assume by way of contradiction there exists a constant $c > 0$ such that $\liminf_{k\to\infty} k\|\rho_k - \rho_{k+1}\|_{L^2(\Omega)}^2 = c$.
	In particular, $\|\rho_k - \rho_{k+1}\|_{L^2(\Omega)}^2 \geq c/(2k)$ for all sufficiently large $k$.
	Together with~\cref{eq:Descent} and $\underline{\alpha} \leq \alpha_k$, we deduce that $\sum_{k=0}^\infty 1/k < \infty$,
	which is a contradiction.

	\smallskip

	\noindent\textsl{Step 3.} For the final claim, we simply note that the bounded, monotonically decreasing sequence of reals $\left\{F(\rho_k)\right\}$ has a limit $F^\star$. If $F$ is completely continuous, then $\left\{ F(\rho_{k}) \right\}_{k \in K_{\rm w}}$ converges to $F(\rho^\star) = F^\star$.
\end{proof}

In practice, we witness rapid growth of the step sizes for both globalized SiMPL methods; cf.\ \Cref{sec:num}.
The following corollary shows that a quadratic growth condition implies strong convergence of the full sequence of SiMPL-A iterates.

\begin{corollary}
	\label{cor:convergence}
	If $\alpha_k \geq c k^{2(1+\epsilon)}$ for some $c > 0$ and $\epsilon > 0$, then there exists a function $\rho^{\star} \in \mathcal{A}$ such that $\rho_k \to \rho^{\star}$ in $L^2(\Omega)$.

\end{corollary}

\begin{proof}
	The goal of the proof is to show that $\sum_{k=0}^\infty \|\rho_{k+1}-\rho_k\|_{L^2(\Omega)} < \infty$, which implies that $\rho_k$ is Cauchy, and thus $\rho^{\star} \in \mathcal{A}$ exists.
	By~\cref{eq:Descent}, we have that
	\[
		c^2 \sum_{k=0}^\infty k^{2(1+\epsilon)} \|\rho_{k+1}-\rho_k\|_{L^2(\Omega)}^2
		\leq
		\sum_{k=0}^\infty \alpha_k\|\rho_{k+1}-\rho_k\|_{L^2(\Omega)}^2
		< \infty
		\,.
	\]
	This implies that $k^{1+\epsilon}\|\rho_k - \rho_{k+1}\|_{L^2(\Omega)} \to 0$ and, in particular, for every $\delta > 0$ there exists $k_0 > 0$ such that $k^{1+\epsilon}\|\rho_k - \rho_{k+1}\|_{L^2(\Omega)} \leq \delta$ for all $k > k_0$.
	Select one such $\delta$ and $k_0$.
	Then $\sum_{k=k_0}^\infty \|\rho_{k+1}-\rho_k\|_{L^2(\Omega)} \leq \delta \sum_{k=k_0}^\infty 1/k^{1+\epsilon} < \infty$, which is sufficient to conclude that $\rho_k$ is Cauchy in $L^2(\Omega)$.

\end{proof}

\subsection{Analysis of SiMPL-B}\label{sec:simpl-b}
This section considers SiMPL-B, i.e., the SiMPL method with a Bregman backtracking line search \cref{eq:bregman-backtracking}.
The SiMPL-B does not require $\nabla F$ to be Lipschitz continuous as opposed to SiMPL-A.

\begin{theorem}\label{thm:gbb-finite-termination-bregman}

	Let $\rho_k,\rho_{k-1}\in \mathbullet{\mathcal{A}}$.
	Suppose that $F$ is relatively smooth to $\varphi$ with constant $\gamma$.
	Then initial step size $\alpha_{k,0}\geq\gamma$.
	Furthermore, the Bregman backtracking line search \eqref{eq:bregman-backtracking} exhibits finite termination with
	\begin{equation}\label{eq:alpha-bdd-bregman}
		\alpha_k \geq \gamma/2=:\underline{\alpha}.
	\end{equation}
\end{theorem}
\begin{proof}
	Noting that $\varphi$ is convex, the relative smoothness implies that $\alpha^{-1}\varphi-F$ is convex for all $\alpha<\gamma$.
	Then the convexity implies that $D_{\alpha^{-1}\varphi-F}$ is a Bregman divergence.
	Therefore, the line search algorithm terminates with a finite number of iterations, and we also obtain a uniform lower bound $\alpha_k>\gamma/2$.
\end{proof}

We end this section with a result on the convergence behavior of the objective function values and increments in SiMPL-B.

\begin{theorem}\label{thm:b-conv}
	Assume that $F$ is relatively smooth to $\varphi$ with constant $\gamma$ and let $\{\rho_k\}$ be a sequence generated by the SiMPL-B method. Then the following hold
	\begin{enumerate}[label={(\roman*)}]
		\item $\{F(\rho_k)\}$ is a monotonically decreasing sequence,
		\item If there exists a uniform upper bound $\alpha_{\rm max} < \infty$ on the step sizes $\alpha_k$, then we have $\|\rho_{k+1} - \rho_{k}\|_{L^p(\Omega)} \to 0$ for all $p \in [1,\infty)$,

	\end{enumerate}
	Moreover, if $F$ is completely continuous, then $F(\rho_k) \to F(\rho^{\star})$, where $\rho^{\star}$ is the weak limit point of $\left\{\rho_{k}\right\}_{k \in K_{\rm w}}$.
\end{theorem}
\begin{proof}
	Since $\rho_{k+1}$ is the unique minimizer of \eqref{eq:mirror-descent-minimizer}, it satisfies the first-order optimality condition, for all $\rho\in \mathcal{A}$
	\begin{align*}
		\int_\Omega \nabla F(\rho_k)(\rho-\rho_{k+1})\dd x+\frac{1}{\alpha_k}\int_\Omega \big(\varphi^\prime(\rho_{k+1})-\varphi^\prime(\rho_k)\big)(\rho-\rho_{k+1})\dd x\geq 0.
	\end{align*}
	Taking $\rho=\rho_k$ and rearranging terms, we obtain
	\begin{align*}
		\int_\Omega \nabla F(\rho_k)(\rho_{k+1}-\rho_k)\dd x\leq -\frac{1}{\alpha_k}\int_\Omega \big(\varphi^\prime(\rho_{k+1})-\varphi^\prime(\rho_k)\big)(\rho_{k+1}-\rho_k)\dd x.
	\end{align*}
	Then the termination condition for the backtracking line search implies
	\begin{align*}
		\begin{aligned}
			F(\rho_{k+1}) & \leq F(\rho_k)+\int_\Omega \nabla F(\rho_k)\big(\rho_{k+1}-\rho_k\big)\dd x+\frac{1}{\alpha_k}D_\varphi(\rho_{k+1},\rho_k)       \\
			              & \leq F(\rho_k)-\frac{1}{\alpha_k}\int_\Omega \big(\varphi^\prime(\rho_{k+1})-\varphi^\prime(\rho_k)\big)(\rho_{k+1}-\rho_k)\dd x \\
			              & \qquad +\frac{1}{\alpha_k} \int_\Omega \varphi(\rho_{k+1})-\varphi(\rho_{k})- \varphi^\prime(\rho_k)(\rho_{k+1}-\rho_k)\dd x     \\
			              & =F(\rho_k)-\frac{1}{\alpha_k}D_\varphi(\rho_k,\rho_{k+1}).
		\end{aligned}
	\end{align*}
	Since the last term on the right-hand side is negative, we have strictly monotonically decreasing objective function values. As in the proof of \Cref{thm:conv} we can deduce the
	existence of some real number $F^{\star}$ such that $F(\rho_k) \to F^{\star}$. Moreover, assuming $F$ is completely continuous, then $F^{\star} = F(\rho^{\star})$,
	where $\rho^{\star}$ is the weak limit point of $\left\{\rho_{k}\right\}_{k \in K_{\rm w}}$.

	Continuing, $F(\rho_k) \to F^{\star}$ and $D_{\varphi}(\rho_k,\rho_{k+1}) \le \alpha_{\rm max}(F(\rho_k) - F(\rho_{k+1}))$ imply $D_{\varphi}(\rho_k,\rho_{k+1}) \to 0$ as $k \to +\infty$.
	Adapting \cite[Proposition 2.3]{Borwein1991} to our setting, we now have $\|\rho_{k+1} - \rho_{k}\|_{L^p(\Omega)} \to 0$ for all $p \in [2,\infty)$.
	Norm convergence extends to all $p \in [1,\infty)$ because $\Omega$ is a bounded domain.
\end{proof}

\section{Numerical experiments}\label{sec:num}
We provide numerical experiments to verify the convergence properties of SiMPL-A and SiMPL-B proven above.

All numerical experiments are based on minimizing the compliance of a two-dimensional cantilever beam \cite{lazarov2011-filter}. In our notation, this means that
\[
	\widehat{F}(\tilde{\rho},u) := \int_{\Omega} f\cdot u \dd x,
\]
where $f \in L^2(\Omega)$ is a bulk forcing term.

In this setting, the solution to the elasticity problem~\cref{eq:lin-elas} is the negative of the solution to the adjoint problem~\cref{eq:adj}.
Therefore, while objective function evaluation (e.g., computing each $F(\rho_k)$) involves solving two PDEs, in particular,~\cref{eq:filt-eq,eq:lin-elas}, gradient evaluation (e.g., computing each $\nabla F(\rho_k)$) only requires solving one additional PDE, namely,~\cref{eq:adj_filter}.
All experiments were implemented in the MFEM library in C++ \cite{mfem2021,mfem2024}.

\subsection{Problem parameters}
In our demonstrative example, we set $\Omega = (0,3)\times(0,1) \subset \mathbb{R}^2$, with zero-displacement boundary conditions $u = 0$ on $\Gamma_0 = [0,3]\times\{0\} \subset \partial \Omega$ and zero-traction boundary conditions on $\partial\Omega \setminus \Gamma_0$; cf.~\cref{eq:lin-elas}.
We then apply a distributed load at the right-hand middle point $x_0 = (2.9, 0.5)$:
\begin{align*}
	f = \begin{cases}(0,-1)^T&\text{ if }\|x-x_0\|_{\ell^2}\leq 0.05,\\ 0&\text{ otherwise}.\end{cases}
\end{align*}

Moreover, we select the void material density $\rho_0=10^{-6}$, the penalty parameter $s=3$, the volume fraction $\theta=0.5$, and the filter radius $r_{\min} = 0.05$ so that $\epsilon = 0.05/2\sqrt{3}$.

\subsection{Initial step size and stopping criteria}
There are many possible stopping criteria.
Since SiMPL methods provide feasible iterates, it is enough to observe the complementary condition introduced in \cref{prop:compl}.
To avoid performing a backtracking line search from a converged solution, we define the following approximated Lagrange multiplier similarly to \eqref{eq:lagrange}:
\begin{align*}
	\tilde{\lambda}_k=(\tilde{\psi}_{k+1}-\psi_k)/\alpha_{k-1}
	\,,~~
	\text{ where }
	\tilde{\psi}_{k+1}=\psi_k-\alpha_{k-1}\nabla F(\rho_k)+\tilde{\mu}
\end{align*}
and $\tilde{\mu}$ solves \eqref{eq:volume-correction} with $\alpha=\alpha_{k-1}$.
We propose the following KKT error estimator to measure the violation of the complementary condition~\eqref{eq:lagrange}:
\begin{equation}\label{eq:kkt_tol}
	\texttt{KKT}_k:= \|\tilde{\lambda}_k - \min\{0, \rho_k +\tilde{\lambda}_k\} - \max\{0, \rho_k - 1 + \tilde{\lambda}_k\}\|_{L^1(\Omega)},
\end{equation}
and we choose to stop our implementations when $\texttt{KKT}_k \leq \texttt{tol}:=10^{-5}$.

For this choice of estimator and stopping criteria, we observed both mesh- and order-independent behavior after discretization; i.e., the number of iterations $k$ needed to reach $\texttt{KKT}_k \leq \texttt{tol}$ was nearly identical across mesh sizes and polynomial degrees.
Following \Cref{rem:mu-solve}, we used the Illinois algorithm for the volume projection. The average number of iterations was approximately 10, and the computational cost was negligible compared to the gradient step.
As the formula~\eqref{eq:alpha-init} for the initial step size guess $\alpha_{k,0}$ is not well-defined at the first iteration $k=0$, we set $\alpha_{0,0}=1/\|\nabla F(\rho_0)\|_{L^\infty(\Omega)}$.

\subsection{Experiment 1: Mesh-independence} 
\label{ssub:experiment_1}

\begin{figure}
	\centering
  \includegraphics[width=0.297\textwidth]{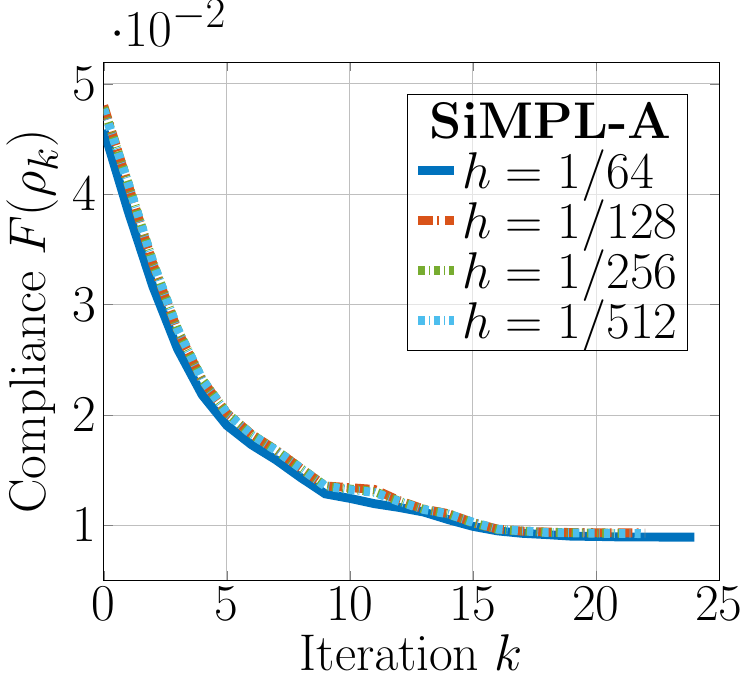}
  \includegraphics[width=0.325\textwidth]{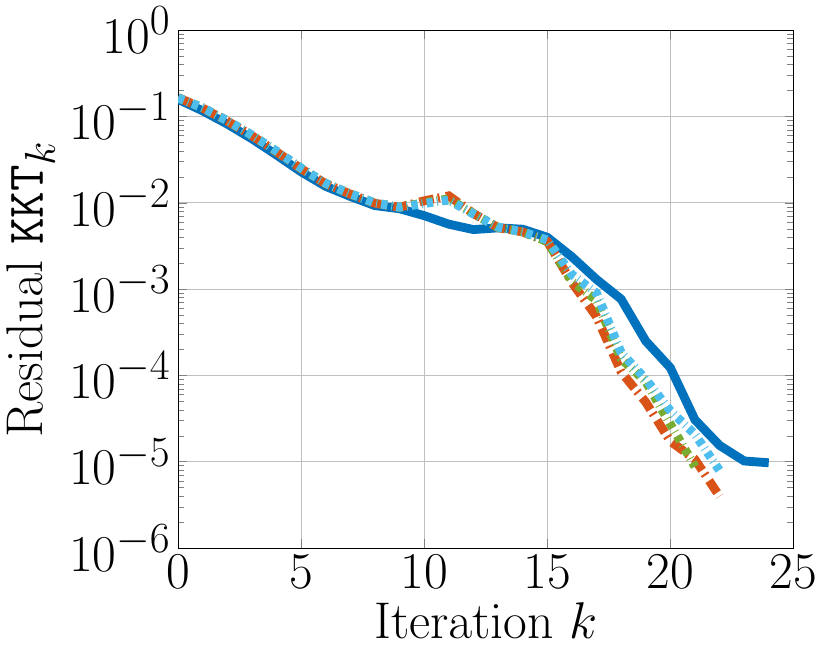}
  \includegraphics[width=0.3125\textwidth]{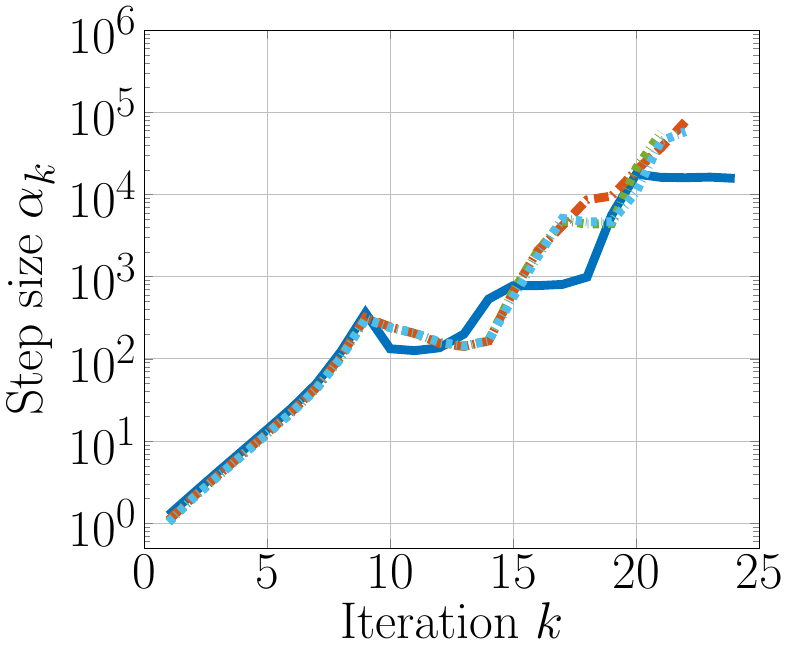}\\[1pt]
  \includegraphics[width=0.297\textwidth]{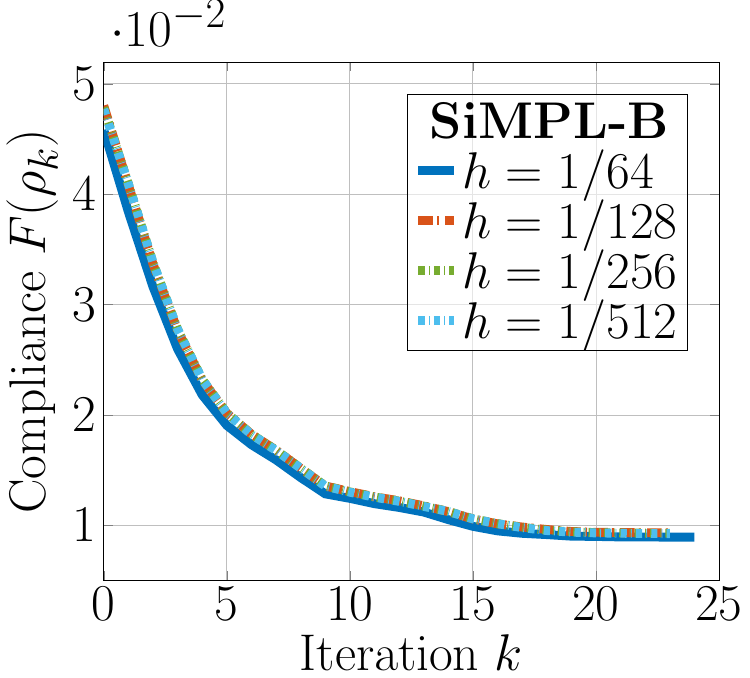}
  \includegraphics[width=0.325\textwidth]{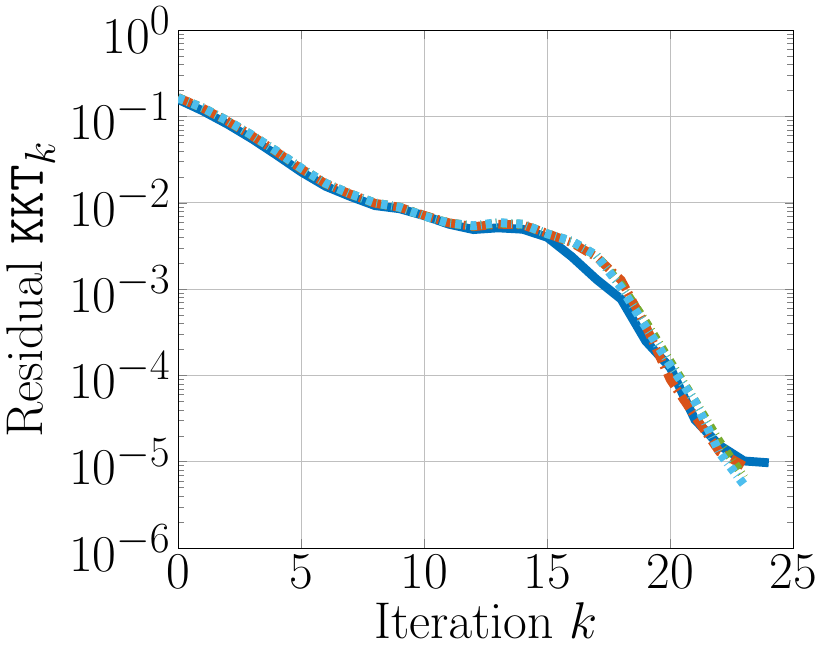}
  \includegraphics[width=0.3125\textwidth]{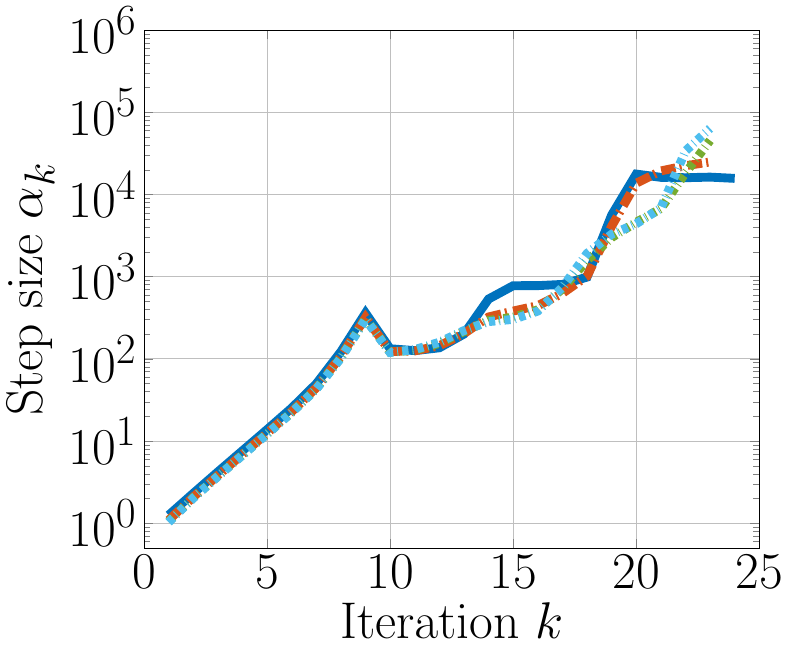}
	\caption{Experiment 1. Compliance (left), KKT residual (center), and step size (right) for $h=1/64$, $1/128$, $1/256$, and $1/512$. SiMPL-A (top row) and SiMPL-B (bottom row) exhibit mesh-independent behavior.}
	\label{fig:cantilever-plots}
\end{figure}
\begin{figure}
	\centering
  \hspace{-2.3em}
	\subfloat[][$\rho_h$ with $h=1/64$]{\includegraphics[width=0.3\textwidth]{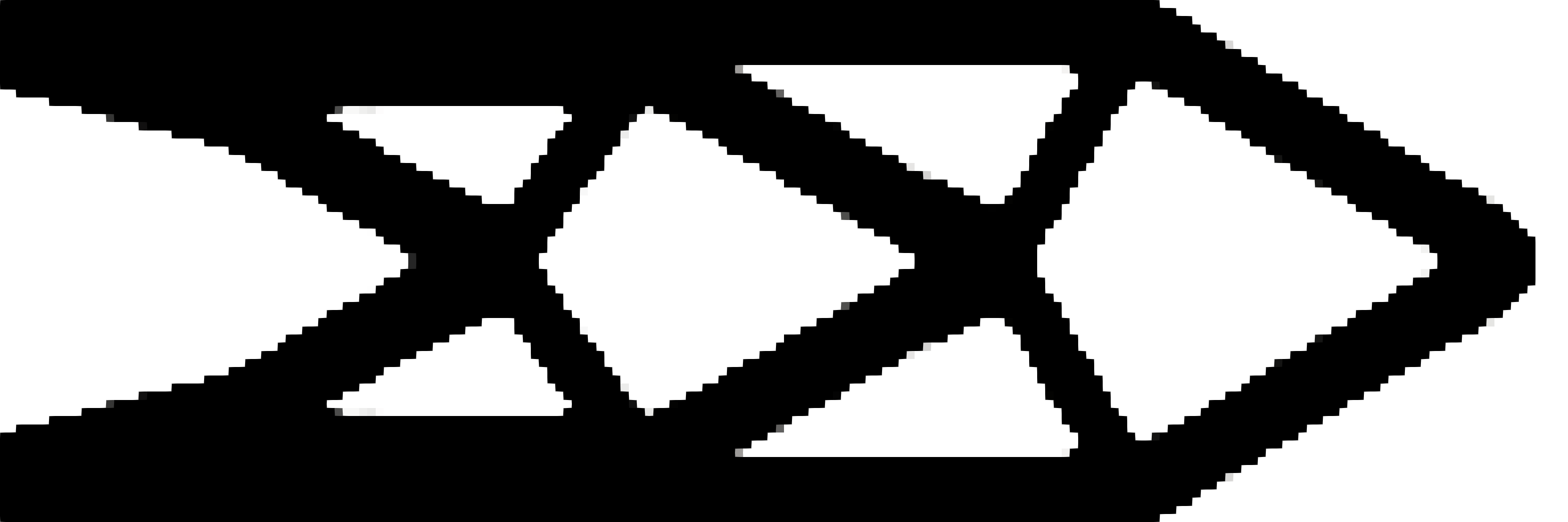}}
	~
	\subfloat[][$\rho_h$ with $h=1/128$]{\includegraphics[width=0.3\textwidth]{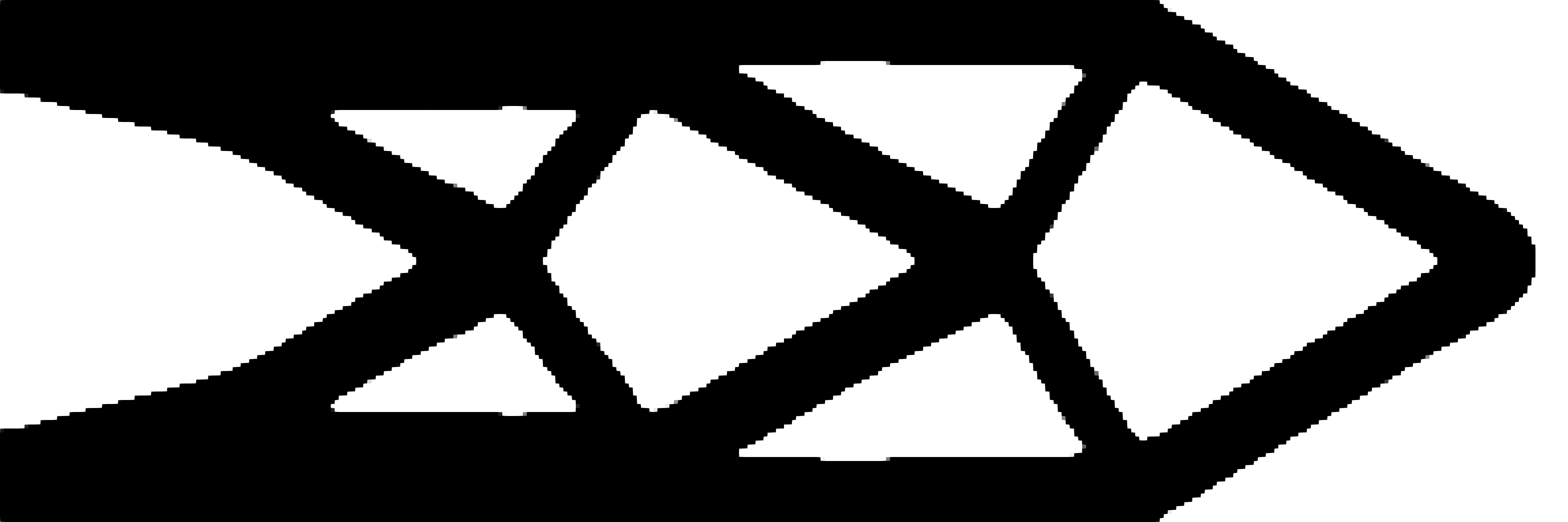}}
	~
  \subfloat[][$\rho_h$ with $h=1/256$]{\includegraphics[width=0.3\textwidth]{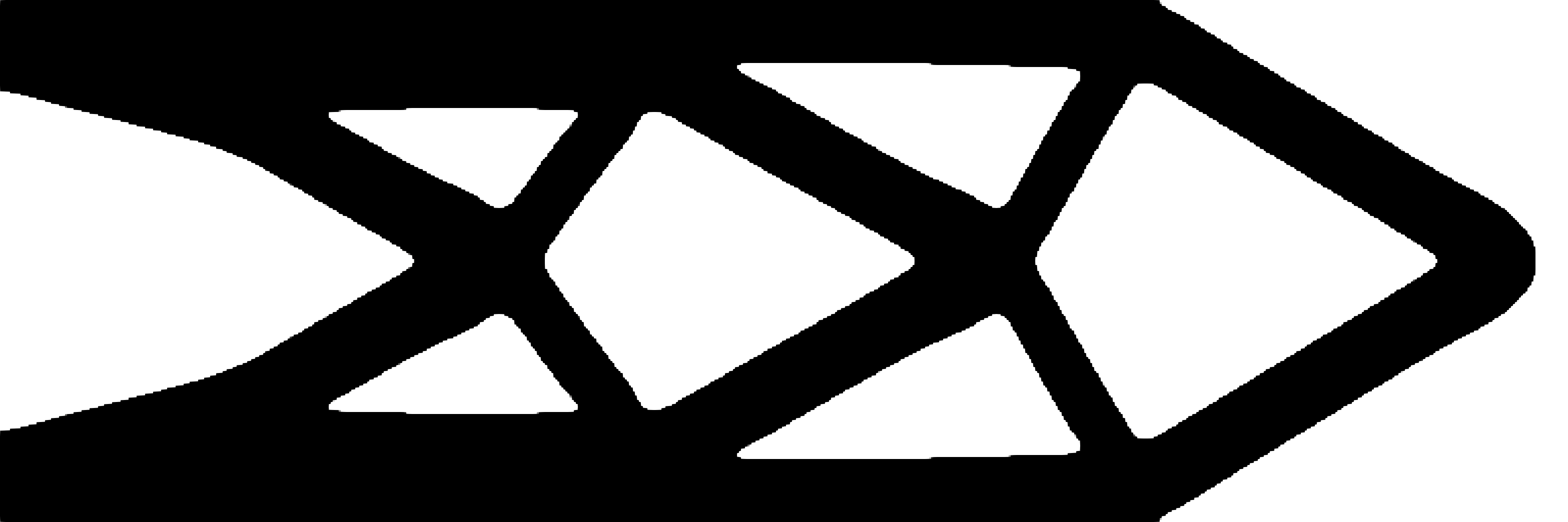}}\\
  \vspace{-5.5em}
	\subfloat[][$\tilde{\rho}_h$ with $h=1/64$]{\includegraphics[width=0.3\textwidth]{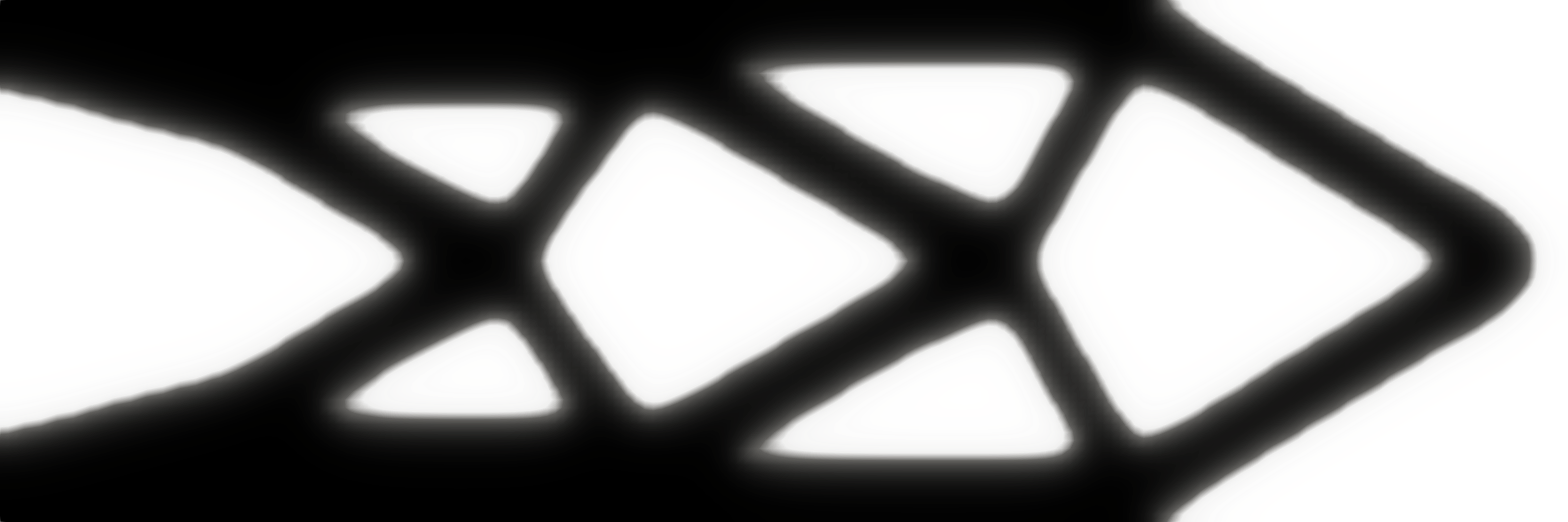}}
	~
	\subfloat[][$\tilde{\rho}_h$ with $h=1/128$]{\includegraphics[width=0.3\textwidth]{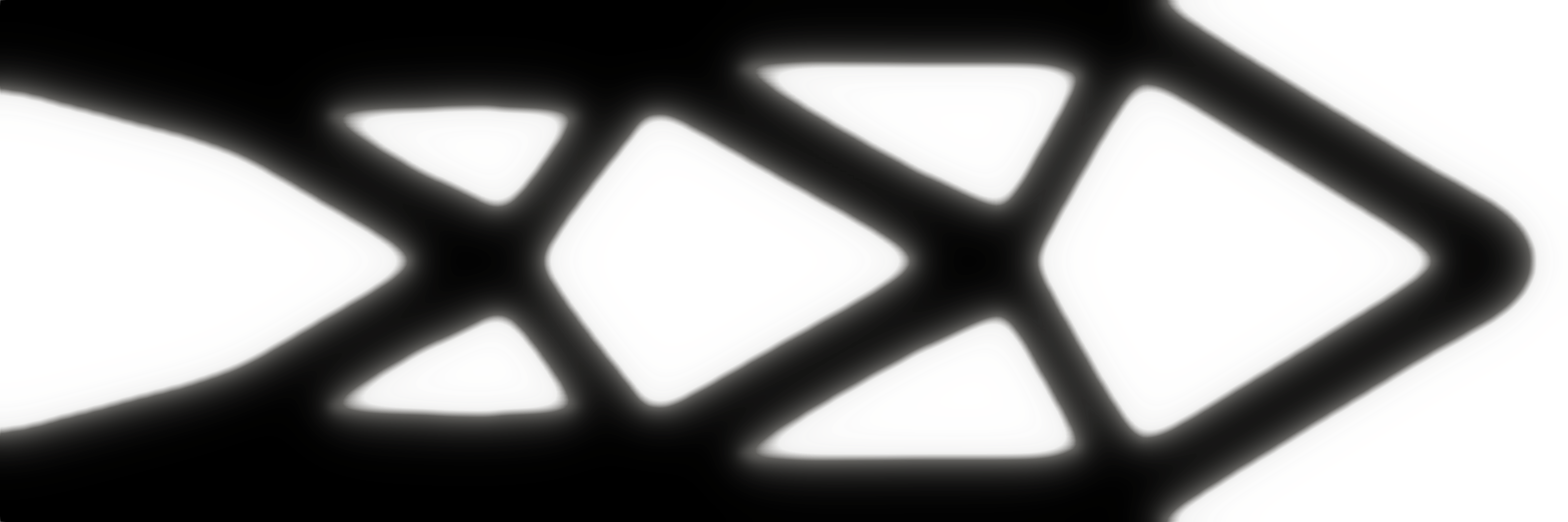}}
	~
	\subfloat[][$\tilde{\rho}_h$ with $h=1/256$]{\includegraphics[width=0.3\textwidth]{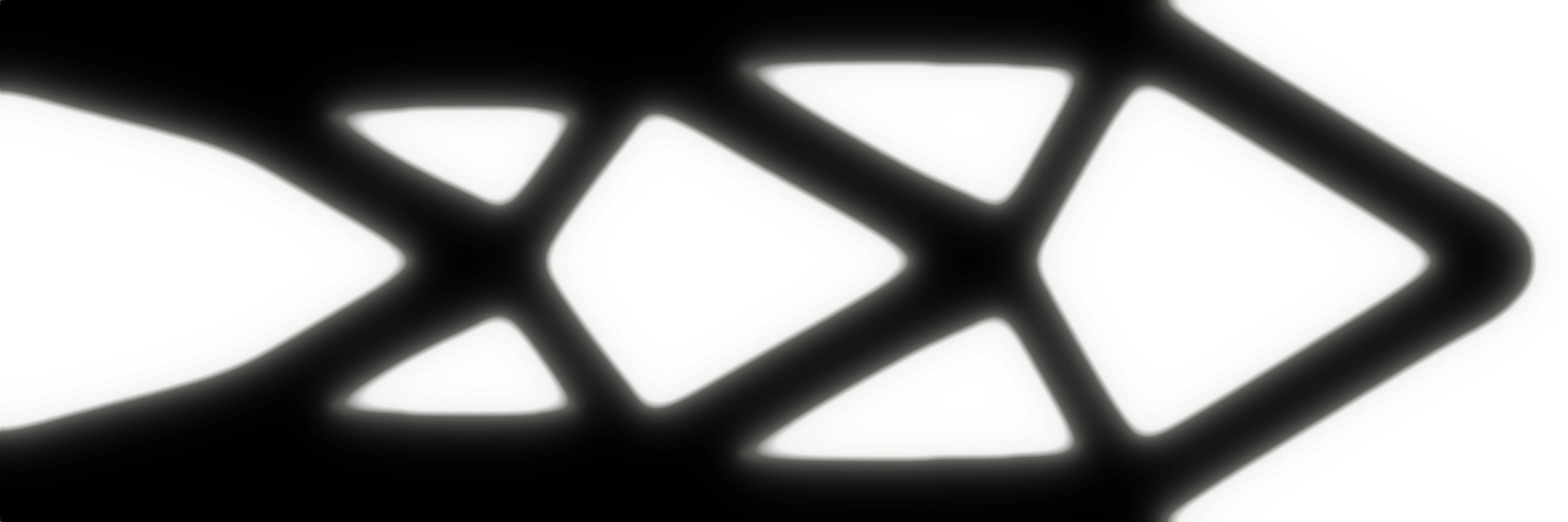}}
  \includegraphics[width=0.05\textwidth]{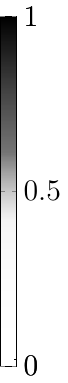}
	\caption{Experiment 1. The discretized design density $\rho_h$ (top) and filtered density $\tilde{\rho}_h$ (bottom) at the final iteration of the SiMPL-B method for $h=1/64$, $1/128$, and $1/256$.
	SiMPL-A yields visually-identical solutions.
	}
	\label{fig:design}
\end{figure}

In this experiment, we used continuous piece-wise linear finite elements over a uniform rectangular grid to discretize the PDEs and piece-wise constant elements to discretize the latent variable $\psi$.
The computed values of the compliance $F(\rho_k)$, KKT residual $\texttt{KKT}_k$, and step size $\alpha_k$ at each iteration are reported in \Cref{fig:cantilever-plots} for different mesh sizes $h>0$.
The corresponding optimized designs are depicted in~\Cref{fig:design}.
From \Cref{fig:cantilever-plots}, we see that the compliance values form a monotonically decreasing sequence $F(\rho_{k+1}) \leq F(\rho_k)$, as expected from \Cref{thm:conv}.
Furthermore, the similar trajectories of the compliance and KKT residual sequences indicate mesh-independence of the method.
We supplement these figures with \Cref{tab:simpl}, which reports the exact number of iterations, backtracking sub-iterations, and PDE solves for each mesh size we tested.
\begin{table}
	\begin{center}
		\begin{tabular}{ |c|c| c| c| c|}
			\hline
			Method                   & $h$             & Iterations & Backtracks & PDE Solves \\
			\hline
			\multirow{4}{*}{SiMPL-A} & 1/64\phantom{0} & 24         & 3          & 78         \\
			                         & 1/128           & 22         & 1          & 68         \\
			                         & 1/256           & 21         & 1          & 65         \\
			                         & 1/512           & 22         & 2          & 70         \\
			\hline
			\multirow{4}{*}{SiMPL-B} & 1/64\phantom{0} & 24         & 3          & 78         \\
			                         & 1/128           & 23         & 2          & 73         \\
			                         & 1/256           & 23         & 2          & 73         \\
			                         & 1/512           & 23         & 2          & 73         \\
			\hline
		\end{tabular}
	\end{center}
	\caption{Experiment 1. The number of iterations, backtracking sub-iterations, and PDE solves to achieve $\texttt{KKT}_k \leq 10^{-5}$ for various mesh sizes $h$.
	\label{tab:simpl}}
\end{table}
To arrive at the number of reported PDE solves, recall that three PDE solves are required for each iteration with two extra PDE solves per backtracking sub-iteration.

The total number of sub-iterations usually falls below 10\% of the total number of mirror descent iterations, suggesting that \eqref{eq:alpha-init} provides a good initial step size guess for the problem considered here.

\subsection{Experiment 2: Order-independence}
\label{ssub:experiment_2_order_independence}
\begin{table}
	\begin{center}
		\begin{tabular}{ |c|c| c| c| c|}
			\hline
			Method                   & ~~$p$~~ & Iterations & Backtracks & PDE Solves \\
			\hline
			\multirow{4}{*}{SiMPL-A} & 1     & 22         & 1          & 68         \\
			                         & 2     & 22         & 1          & 68         \\
			                         & 3     & 23         & 1          & 71         \\
			                         & 4     & 23         & 1          & 71         \\
			\hline
			\multirow{4}{*}{SiMPL-B} & 1     & 23         & 2          & 73         \\
			                         & 2     & 23         & 2          & 73         \\
			                         & 3     & 22         & 2          & 70         \\
			                         & 4     & 23         & 2          & 73         \\
			\hline
		\end{tabular}
	\end{center}
	\caption{Experiment 2. The number of iterations, backtracking sub-iterations, and PDE solves to achieve $\texttt{KKT}_k \leq 10^{-5}$ for various polynomial orders $p$.\label{tab:simpl-order}}
\end{table}

Unlike other methods for topology optimization, which require discretizing the density field $\rho_k$ with finite elements, SiMPL always yields pointwise feasible iterates.
This is because the SiMPL method requires discretizing the latent variable $\psi_k$ with finite elements and then using a sigmoid function to construct the discrete density field $\rho_{k,h} = \sigma(\psi_{k,h})$; cf.\ \Cref{sub:LVMD}.
To demonstrate how this can be leveraged to enable high approximation order topology optimization, we fixed $h = 1/64$ in Experiment 1 and tested discretizations with higher approximation orders $p>0$.
More specifically, letting $p = 0$ denote the polynomial degree of the discretization used in Experiment 1, we use $p = 1$, $2$, $3$, and $4$ respectively denote discretizations using the next four higher-order finite element discretizations in each solution variable.
For instance, $p = 3$ indicates using continuous piece-wise quartic (fourth-order) finite elements to discretize the PDEs, while the latent variable $\psi$ is discretized with piece-wise cubic (third-order) elements.
The results are given in \Cref{fig:cantilever-order,tab:simpl-order}.
\begin{figure}
	\centering
  \includegraphics[width=0.297\textwidth]{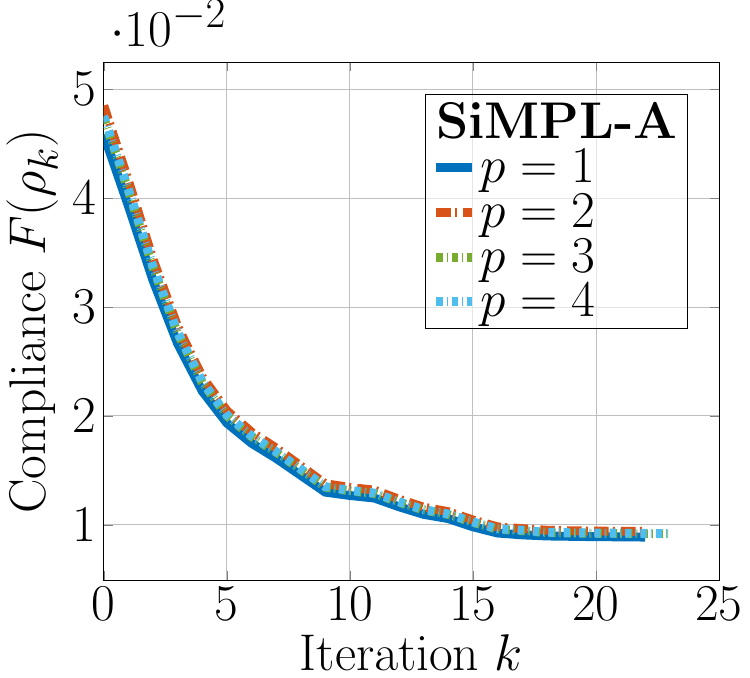}
  \includegraphics[width=0.325\textwidth]{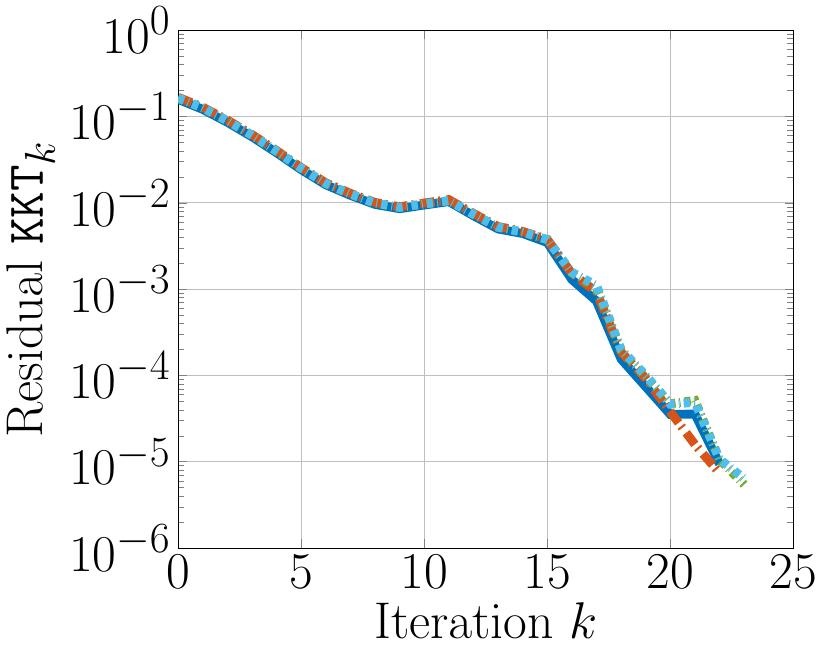}
  \includegraphics[width=0.3125\textwidth]{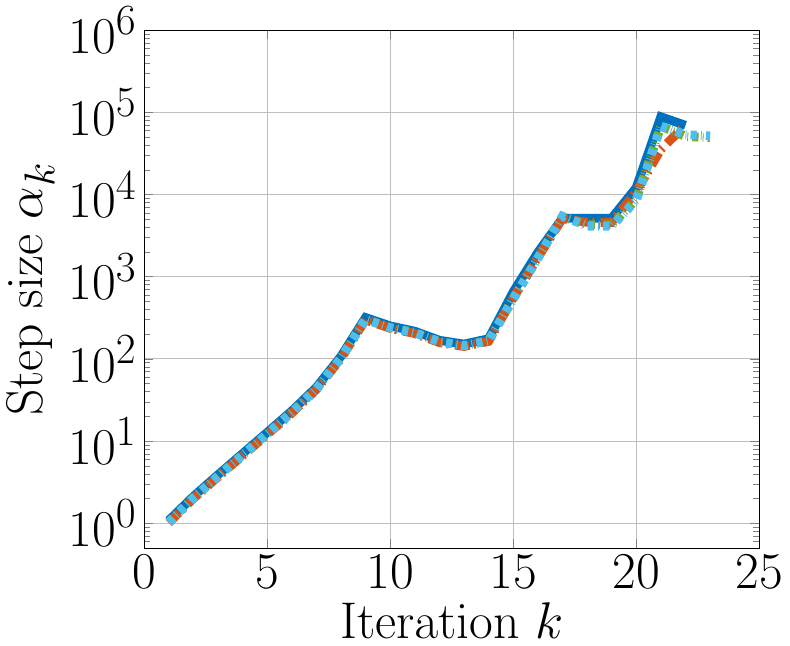}\\[1pt]
  \includegraphics[width=0.297\textwidth]{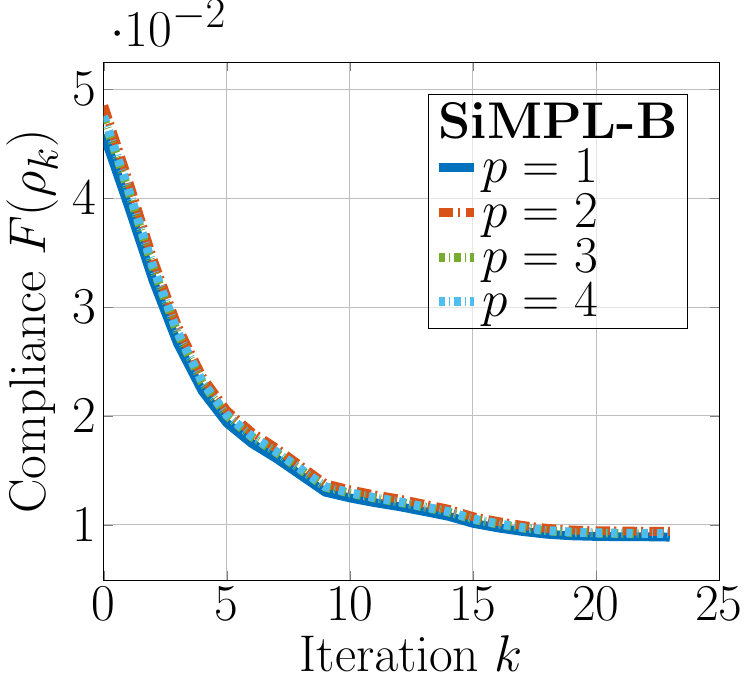}
  \includegraphics[width=0.325\textwidth]{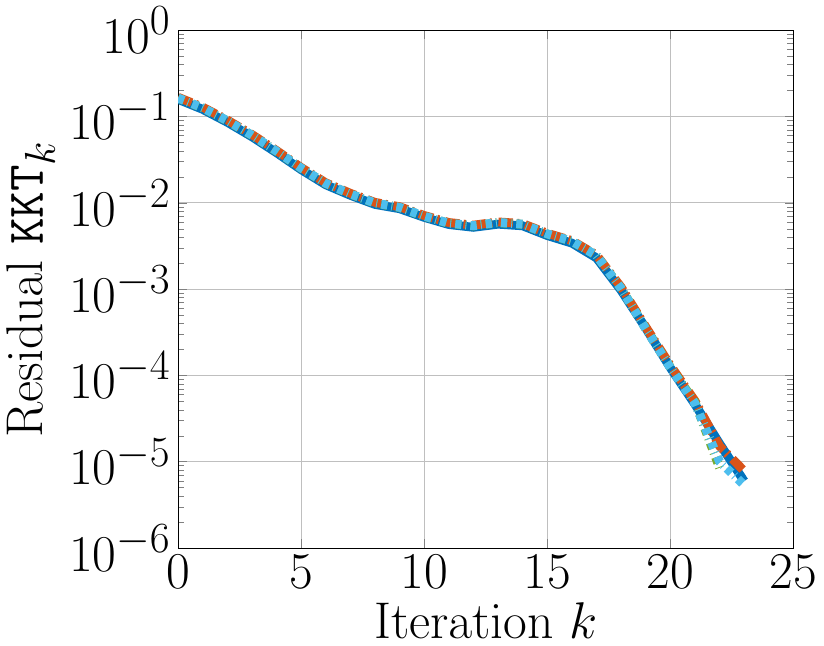}
  \includegraphics[width=0.3125\textwidth]{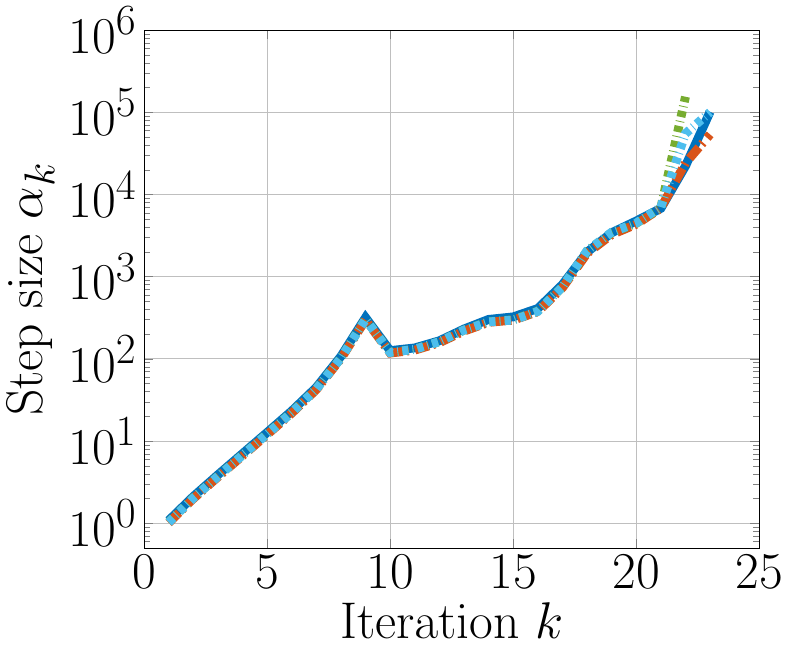}
	\caption{Experiment 2. Compliance (left), KKT residual (center), and step size (right) for polynomial degrees $p=1, 2, 3$ and $4$ on a uniform rectangular grid $h=1/64$.  SiMPL-A (top row) and SiMPL-B (bottom row) exhibit order-independent behavior.}
	\label{fig:cantilever-order}
\end{figure}
As in Experiment 1, the number of iterations to reach $\texttt{KKT}_k \leq 10^{-5}$ remained nearly constant, regardless of the polynomial degree.
Also, the number of backtracking line search sub-iterations appear to remain constant between orders, indicating order-independence of the SiMPL method.

\begin{remark}
  \label{rem:filter-dmp}

We note that while each discrete density field $\rho_{k,h}$ satisfies the bound constraints $0 \leq \rho_{k,h} \leq 1$ everywhere in $\Omega$, the discrete filtered density $\tilde{\rho}_{k,h}$ may violate these bound constraints since the standard finite element method we use to solve~\cref{eq:filt-eq} may not satisfy a discrete maximum principle. We can obtain a bound-preserving filtered density following \cite{keith2023proximal} or using other methods that satisfy the discrete maximum principle.
For simplicity, in Experiment 2, we clipped the filtered density $\tilde{\rho}_{k,h}$ if it overshot the bound constraints $0$ and $1$ and witnessed no adverse side-effects.

\end{remark}

\section{Conclusion}
\label{sec:Conclusion}
In this work, we derive and analyze the SiMPL method for topology optimization with two possible globalization strategies based on backtracking line search algorithms.
These strategies lead to two versions of the method referred to as SiMPL-A and SiMPL-B.
The latter version is completely parameter-free, while the former involves a tunable parameter $c_1=10^{-4}$ whose optimal value depends on the underlying problem.
By approximating the latent variable $\psi$ instead of the primal design density $\rho$, the SiMPL method is numerically stable while producing pointwise feasible iterates on the \textit{discrete level} due to the relation $\rho_{h} = \sigma(\psi_h)$.
As expected, the use of adaptive step size selection strategies greatly accelerates what would normally amount to a first-order method and, in fact, leads to almost second-order method behavior.
We suspect that this is due to the choice of $\alpha_k$ and its relation to the classic Barzilai--Borwein algorithm, which in essence is a type of quasi-Newton method, and the fact that the line search strategies often accept $\alpha_k$ immediately.
The SiMPL method can be extended naturally to high-order approximations without losing the pointwise bound preserving property.
Since we allowed for the objective function to be rather arbitrary, SiMPL is flexible enough to be applied to a variety of topology optimization problems such as self-weighted compliance minimization or compliant mechanism, as has been demonstrated in the companion paper \cite{simplapp}, where it is also shown to outperform the most popular first-order methods in the topology optimization literature.

\bibliographystyle{siamplain}
\bibliography{main.bbl}
\end{document}